\documentclass{article}
\usepackage{a4wide}
\usepackage{amsmath, amssymb, amsbsy, amsfonts, mathptmx, amsthm}
\usepackage{lmodern}
\usepackage{euscript}
\newcommand{\dv}{\operatorname{div}}
\newcommand{\Rb}{\mathbb{R}}
\newcommand{\Kb}{\mathbb{K}}
\newcommand{\xb}{\mbox{\bf x}}
\newcommand{\dert}[1]{\frac{\partial #1}{\partial t}}

\newcommand{\ql}{\mathbf{q}_l}
\newcommand{\qg}{\mathbf{q}_g}
\newcommand{\gb}{\mathbf{g}}

\newcommand{\Hb}[1]{(\textsf{A.#1})}
\newcommand{\Pb}[1]{(\textsf{P.#1})}
\newtheorem{theorem}{Theorem}
\newtheorem{lemma}{Lemma}
\newtheorem{proposition}{Proposition}
\newtheorem{corollary}{Corollary}
\newtheorem{remark}{Remark}

\title{Two-phase two-component flow in porous medium in low solubility regime}
\author{Mladen Jurak\thanks{Faculty of Science, University of Zagreb ({\tt jurak@math.hr}).}
\and Ivana Radi\v si\'c\thanks{Faculty of Mechanical Engineering and Naval Architecture, University of Zagreb ({\tt iradisic@fsb.hr}).}  
\and  Ana \v Zgalji\'c Keko\thanks{Faculty of Electrical Engineering and Computing, University of Zagreb ({\tt ana.zgaljic@fer.hr}).}
}

\begin{document}

\maketitle

\begin{abstract}
We study a system of equations governing liquid and gas
flow  in porous media. The gas phase is homogeneous while the liquid phase is composed of a 
liquid component and dissolved gas component. It is assumed that the gas component is weakly soluble in the liquid.
We formulate a weak solution of the initial-boundary value problem and prove 
the existence theorem by passing to the limit in regularizations of the problem.
Hypothesis of low solubility is given precise mathematical meaning.
\end{abstract}

\tableofcontents

\section{Introduction}
\label{sec:introduction}

The simultaneous flow of immiscible fluids in porous media occurs in
a wide variety of applications such as   unsaturated groundwater flows and flows in underground
petroleum reservoirs. More recently, multiphase flows have generated
serious interest among engineers concerned with  nuclear waste management and in particular
 the migration of gas through the near
field environment and the host rock, which in the simplest case 
 involves two components, water and  hydrogen and
two phases, liquid and gas (see \cite{forge}). 
In this application the gas component (hydrogen) is weakly soluble in water
but the solubility is still highly important for long term gas migration and the repository pressurization.

An important consideration in the modeling of fluid  flow with mass exchange between phases
is the choice of the primary variables that define the thermodynamic state of 
the fluid system, \cite{WuFo}. When a phase appears or disappears, the set of appropriate thermodynamic
variables may change. In mathematical analysis of the two-phase, two-component model presented 
in this article  we chose a formulation based on persistent variable approach \cite{TA02-BJS, CG-BJS, EDF}. 
Namely, we  use two pressure-like variables capable of describing the fluid system in both one-phase
and two-phase regions.

The mathematical theory of incompressible, immiscible and isothermal  two-phase flow through porous media is developed
in extensive literature and summarized in several monographs \cite{ant-kaz-mon1990, ChJa, GG-MM-1996} 
and articles \cite{Chen-I,Chen-II}. An analysis of nonisothermal immiscible incompressible model is presented in \cite{AJPP}.
Development of mathematical theory for compressible, immiscible 
two-phase flow  started with the work of Galusinski and Saad \cite{CG-MS1, CG-MS2, CG-MS3} and is further
developed in \cite{AJ, AJZK, AJV, ba-lp-ap, Khalil-Saad, Khalil2, schweizer1, hilhorst}. For the two-phase
compositional flow model there are much less publications. First, incomplete results were obtained in 
\cite{FaridEx} and \cite{Mik}. More complete two-phase, two-component models were considered in articles
 \cite{saad1} and \cite{saad2}.
 In  \cite{saad1} the authors replace the phase equilibrium by the first order chemical reactions which are 
 supposed to model the mass exchange between the phases. 
  In \cite{saad2}  the phase equilibrium model is taken into account but the degeneracy of the diffusion terms 
 is eliminated by some non-physical assumptions.
 As the diffusion terms in the flow equations are multiplied by the liquid 
 saturation they can be arbitrary small (see (\ref{diff-fluxes})) and they do not add sufficient 
 regularity to the system. In this work this degeneracy of the diffusive terms is compensated by the low
 solubility of the gas component in the liquid phase which keeps the liquid phase composed mostly of the
 liquid component (water). This compensation allows us to treat the complete two-phase two-component model 
 without any unphysical assumptions on the diffusive parts of the model.  
 
 The outline of this paper is as follows.
  In Section 2 we give a short description  of the physical and mathematical model of two-phase, two-component 
  flow in porous medium  considered in this study. We also introduce the global pressure that plays an important role in 
 mathematical study of the model, general assumptions on the data and some auxiliary results. 
 In Section 3 we present the main result of this paper, the existence of a weak solution to an initial 
 boundary value problem for considered two-phase, two-component   flow model.  This theorem is proved 
 in the following sections. In Section 4 we regularize the system and discretize the time derivatives, 
 obtaining thus a sequence of elliptic problems. In Section 5 we prove the existence theorem for 
 the elliptic problems by an application of the Schauder fixed point theorem. In this section we perform 
 further regularizations and apply special test functions which lead to the energy estimate on which
 the existence theorem is based. In Sections 6 and 7 we eliminate the time discretization and 
 the initial regularization of the system by passing to zero in the small parameters. At the limit 
 we obtain a solution of the initial  two-phase, two-component   flow model.

\section{Mathematical model}
\label{sec:model}

We consider herein a porous medium saturated with a fluid composed of 2 phases,
 \textit{liquid} and \textit{gas}, and according to the application we have in mind, we
 consider the fluid as a mixture of two
 components: a \textit{liquid component} which does not evaporate  and  a \textit{low-soluble component}
 (such as hydrogen) which is present mostly in the gas phase and dissolves in the liquid phase. 
  The porous medium is assumed to be rigid and in the thermal equilibrium, 
  while the liquid component is assumed incompressible.

The two phases are denoted by indices, $l$ for liquid, and $g$ for
gas. Associated to each phase $\sigma\in\{l,g\}$, we have the phase pressures $p_\sigma$, the phase
saturations $S_\sigma$, the phase mass densities $\rho_\sigma$ and
the phase volumetric fluxes $\mathbf{q}_\sigma$  given by the {\em  Darcy-Muskat law } (see \cite {Ours1, Ours2, ChJa}):
\begin{equation}
 \ql = - \lambda_l(S_l)\Kb(\xb) \left( \nabla p_l -\rho_l \gb\right),\quad
          \qg  = - \lambda_g(S_l) \Kb(\xb) \left( \nabla p_g -\rho_g \gb\right),
	  \label{darcy}
\end{equation}
where $ \Kb(\xb)$ is the absolute permeability tensor,
$\lambda_\sigma(S_l)$ is  the $\sigma-$phase relative mobility
function, and $\gb$ is the gravity acceleration. There is no void space in the porous medium, 
meaning that the phase saturations  satisfy 
$
S_l + S_g = 1 .   
$
The phase pressures are connected through the  {\em capillary pressure law} (see \cite {Ours1,ChJa}) 
\begin{equation}
 p_c(S_l) = p_g -p_l,
 \label{capillary}
\end{equation}
where the function $p_c$ is a strictly
decreasing function of the liquid saturation, $ p_c'(S_l) < 0 $.

In the gas phase, we neglect the liquid component vapor such that the gas mass density  depends 
only on the gas pressure: 
\begin{equation}
\rho_g = \hat\rho_g(p_g), \label{ideal}
\end{equation}
where in the case of the ideal gas law we have $\hat\rho_g(p_g) = C_v p_g$
with $C_v = M^h/(RT)$, where $M^h$ is molar mass of the gas component, $T$ is the temperature and $R$ is the universal gas constant.

The liquid component will be denoted by upper index $w$ (suggesting \textit{water}) and
the low-soluble gas component will be denoted by upper index $h$  (suggesting \textit{hydrogen}). 
In order to describe the quantity of the gas component dissolved in the liquid we introduce mass 
concentration $\rho_l^h $ which gives the mass of dissolved gas component in the volume of the liquid mixture. To simplify notation we will denote $\rho_l^h$ by $u$.
 The assumption of thermodynamic equilibrium  leads to functional dependence:
\begin{equation}
 u = \hat u(p_g),   \label{Henry}
\end{equation}
if the gas phase is present. In the absence of the gas phase $ u $ must be considered as an
independent variable. 
If the Henry law is applicable then the function $\hat u$ can be taken as a linear function  $u = C_h p_g$,
  where $C_h = H M^h$ and $H$ is the Henry law constant.
We suppose that the function $p_g \mapsto \hat u(p_g)$ is defined  and invertible on $ [0,\infty)$
and therefore we can express the gas pressure as a function of  $u$,
\begin{align}
  p_g = \hat{p}_g(u),
  \label{HenryInverz}
\end{align}
where $\hat{p}_g$ is the inverse of $\hat u$. 

For liquid density, due to hypothesis of small solubility 
and liquid incompressibility we may assume constant liquid component mass concentration, i.e.:
\begin{equation}
    \rho_l^w = \rho_l^{std},
\label{const-w}
\end{equation}
where $\rho_l^{std}$ is the standard liquid component mass density (a constant). The liquid mass density is then:
$\rho_l = \rho_l^{std} + u$.

Finally, the  mass conservation for each component leads to the following differential
equations:
\begin{align}
   \rho_l^{std} \Phi\dert{S_l}  + \dv \left(  \rho_l^{std}  \ql   +\mathbf{j}_l^w \right) = {\cal F}^w, \label{eq-1} \\
\Phi\dert{}\left(   u S_l +   \rho_g S_g\right) + \dv \left(  u \ql +  \rho_g  \qg +\mathbf{j}_l^h \right)
= {\cal F}^h,\label{eq-2}
\end{align}
where the phase flow velocities, $\ql$ and $\qg$, are given by the
Darcy-Muskat law (\ref{darcy}), ${\cal F}^k$ and $\mathbf{j}_l^k$, $k\in\{w,h\}$,  
are respectively  the $k-$component source terms and the diffusive flux in the liquid
phase. 
The diffusive fluxes are given by the Fick law
which can be expressed 
through the gradient of the mass fractions $X_l^h= u /\rho_l$ and $X_l^w=\rho_l^w/\rho_l$
as in \cite{Ours2, TA02-BJS}:
\begin{align}
\mathbf{j}_l^h =  -\Phi S_l D \rho_l \nabla X_l^h,\quad \mathbf{j}_l^w =  -\Phi S_l D \rho_l \nabla X_l^w  ,
\label{diff-fluxes}
\end{align}
where $D$ is a molecular diffusion coefficient of dissolved gas in the
liquid phase, possibly corrected by the tortuosity factor of
the porous medium (see \cite{Ours2}). Note that we have $X_l^h + X_l^w =1$, leading to $\mathbf{j}_l^h +\mathbf{j}_l^w = 0$.
The source terms  ${\cal F}^w$ and  ${\cal F}^h$ will be taken in the usual form:
\begin{align}
  {\cal F}^w = \rho_l^{std} F_I - \rho_l^{std} S_l F_P,\quad  {\cal F}^h = - (u S_l +   \rho_g S_g) F_P ,\quad F_I, F_P \geq 0,
  \label{source-terms}
\end{align}
where $F_I$ is the rate of the fluid injection and $F_P$ is the rate of the production.
For simplicity we  supposed that only wetting phase is injected, while composition of produced fluid is not
a priori known.

In the model described here the liquid phase is always present but the gas phase can disappear and reappear 
in certain regions of the porous domain. Mathematical analysis of this model presented in this article is based on 
the {\em persistent variables} approach. Several sets of persistent variables are proposed in the literature
for this model (see \cite{TA02-BJS, CG-BJS, EDF}). We chose approach taken in \cite{CG-BJS} and in \cite{EDF}
which consists in taking $p_l$ and $u$ as variables that can describe the fluid state in both one-phase and two-phase 
regions. In particular, we follow the approach proposed in  \cite{EDF} which 
consists in using relation (\ref{HenryInverz}) to define the gas pressure 
even in the case where the gas phase is nonexistent.  The {\em gas pseudo-pressure}
defined by  (\ref{HenryInverz}) is an artificial variable {\em proportional} to the concentration of the dissolved
gas in the one-phase region and equal to the gas phase pressure in the two-phase region. In that way one avoids 
using directly the concentration of the dissolved gas $u$ as a primary variable and uses more traditional 
gas pressure, suitably extended in one phase region. 

We consider the liquid pressure $p_l$ and the gas pseudo-pressure $p_g$ as primary variables from which we 
calculate several {\em secondary} variables: 
\begin{align}\label{second}
S_l = p_c^{-1}(p_g - p_l), \; S_g = 1-S_l, \; u = \hat{u}(p_g),\; \rho_g = \hat{\rho}_g(p_g), \; \rho_l = \rho_l^{std}+ \hat{u}(p_g).
\end{align}
 Note that in the two-phase region we can    recover
the liquid saturation by inverting the capillary pressure curve, $S_l = p_c^{-1}(p_g - p_l)$. In the one phase region
we  set the  liquid saturation to one, which
amounts to extending the inverse of the capillary pressure curve by one for negative pressures (see \Hb{4}), 
as described in \cite{CG-BJS}. As a consequence we have $0\leq S_l\leq 1$ by properties of the capillary pressure function
(see \Hb{4}).

\subsection{Problem formulation}
\label{sec:3}

Let $\Omega\subset \mathbb{R}^d$, for $d=1,2,3$, be a bounded Lipschitz domain and let $T>0$. 
We assume that $\partial \Omega = \Gamma_D \cup \Gamma_N$, is a regular partition of the boundary with 
$|\Gamma_D| > 0$. 
We consider the following initial-boundary value problem in $Q_T = \Omega\times (0,T)$
 for the  problem (\ref{eq-1})--(\ref{source-terms}) written in selected variables:

\begin{equation}
   \Phi\dert{S_l} - 
	\dv \Big(  \lambda_l(S_l)\Kb(\nabla p_l -\rho_l\gb)  
	        -\Phi S_l \frac{1}{\rho_l} D  \nabla u
	    \Big) +  S_l F_P  =  F_I, \label{eq-1B}
\end{equation}	 
\begin{equation}  
 \begin{split}
	\Phi\dert{}\big(   u S_l  +   \rho_g S_g\big) &-
	    \dv \Big(u \lambda_l(S_l) \Kb(\nabla p_l - \rho_l \gb)
	    + \rho_g \lambda_g(S_l)  \Kb(\nabla p_g -\rho_g \gb)\Big) \\
& -\dv\Big(\Phi S_l \frac{\rho_l^{std}}{\rho_l} D  \nabla u \Big) +  (u S_l + \rho_g S_g) F_P= 0,
\end{split}
\label{eq-2B}
\end{equation}
 with homogeneous Neumann's  boundary condition imposed: on $\Gamma_N$
\begin{equation}
\begin{split}
  \Big(  \lambda_l(S_l)\Kb(\nabla p_l -\rho_l\gb)  -\Phi S_l \frac{1}{\rho_l} D  \nabla u \Big)  \cdot\mathbf{n} = 0,\\
 \Big(u \lambda_l(S_l) \Kb \big(\nabla p_l - \rho_l \gb\big) 
	    + \rho_g \lambda_g(S_g)  \Kb\big(\nabla p_g -\rho_g \gb\big)
 + \Phi S_l \frac{\rho_l^{std}}{\rho_l} D  \nabla u \Big) \cdot\mathbf{n} = 0,
\end{split}
    \label{BC-N}
\end{equation}
on $\Gamma_N$ and 
\begin{equation}
    p_l = 0,\quad p_g = 0,           
    \label{BC-D}
\end{equation}
on $\Gamma_D.$
We impose initial conditions as follows
\begin{align}
p_l(x,0)=p_l^{0}(x),\quad p_g(x,0)=p_g^{0}(x). \label{initial}
\end{align}

All the secondary variables 
 $S_l$, $S_g$, $u$, $\rho_g$ and $\rho_l$ in 
(\ref{eq-1B}), (\ref{eq-2B}), are calculated from $p_l$ and $p_g$ by (\ref{second}). 
The boundary condition $p_g=0$ on $\Gamma_D$ is equivalent to the condition $u=0$
which impose that there is no dissolved gas on the boundary (see \Hb{5}). 

\subsection{The global pressure}

We will use the notion of the global pressure $p$ as given in \cite{ChJa}. 
The global pressure $p$ is defined in connection with the liquid pressure as
\begin{align}
p=p_l+\overline{P}(S_l),\quad \overline P(S_l)= - \int_{S_l}^1 \frac{\lambda_g(s)}{\lambda(s)}p_c'(s) ds .
\label{glrelation}
\end{align}
where 
$
  \lambda(S_l) = \lambda_l(S_l)+\lambda_g(S_l)
$
is the {\em total mobility}. 
From (\ref{glrelation}) and \Hb{4} in Section~\ref{section-MA}
  it follows  that $p_l \leq p$ and $p = p_l$ in the area where $p_g \leq p_l$.
In other words, when the gas pseudo-pressure falls  below the liquid pressure, and only the liquid phase remains,
 then the global pressure coincides with the liquid pressure. 

In the part of the domain where $p_g \geq p_l$ we have second representation of the global pressure,
\begin{align}
p=p_g+\hat P(S_l),\quad \hat P(S_l)= \int_{S_l}^1 \frac{\lambda_l(s)}{\lambda(s)}p_c'(s) ds,
\label{glrelation-g}
\end{align}
but in the domain area where $p_g < p_l$ formula (\ref{glrelation-g}) does not hold true as there 
the global pressure stays equal to the liquid pressure. 
From (\ref{glrelation}) we have a.e. 
\begin{align}\label{glrelation1-grad}
\nabla p_l=\nabla p -\frac{\lambda_g(S_l)}{\lambda(S_l)}\nabla p_c(S_l),
\end{align}
and from (\ref{glrelation-g}) it follows 
\begin{align}\label{glrelation2-grad}
\nabla p_g=\nabla p +\frac{\lambda_l(S_l)}{\lambda(S_l)}\nabla p_c(S_l),
\end{align}
in the part of $Q_T$ where $p_g \geq p_l$. 
By introducing the functions,
\begin{align}
  \label{alpha-beta}
\gamma(S_l) = \sqrt{\frac{\lambda_l (S_l)\lambda_g(S_l)}{\lambda(S_l)} }, \ \ \alpha(S_l)=-\gamma(S_l) p_c'(S_l) \
\ \mathrm{and}  \ \ \beta(S_l)=\int_{0}^{S_l}\alpha(s)\,ds,
\end{align}
we can write formally
\begin{equation}
  \label{fund-gl1a}
	\begin{split}
\lambda_g(S_l)\nabla p_g=\lambda_g(S_l)\nabla p - \gamma(S_l) \nabla \beta(S_l),\quad
\lambda_l(S_l)\nabla p_l=\lambda_l(S_l)\nabla p + \gamma(S_l)\nabla \beta(S_l).
\end{split}
\end{equation}
These two equations hold true a.e in $Q_T$ if $p,\beta(S_l)\in H^1(\Omega)$ for a.e. $t\in (0,T)$. 
From here we easily conclude that the following fundamental equality holds:
\begin{lemma}
  \label{lema-fund-gl1}
Let  $p$ and $\beta(S_l)$ be $H^1(\Omega)$ functions for a.e. $t\in (0,T)$. Then 
equations (\ref{fund-gl1a}) hold a.e. in $Q_T$ as well as equation: 
  \begin{align}
\lambda_l(S_l)\Kb\nabla p_l\cdot \nabla p_l +\lambda_g(S_l)\Kb\nabla p_g\cdot \nabla p_g = 
\lambda(S_l)\Kb\nabla p\cdot \nabla p + \Kb\nabla \beta(S_l)\cdot \nabla\beta(S_l).
\label{fund-gl1}
\end{align}
\end{lemma}

\begin{lemma}\label{lema-GP-bounds}
 Under assumptions   \Hb{4} and \Hb{8} there exists a constant $C>0$ such that  the following bounds hold true:
 \begin{align}
    p_g^+ \leq |p| + C, \quad |S_l p_l | \leq |p| + C,\quad |S_g p_g | \leq |p| + C,\quad
     p_l  \leq p \leq  \max(p_l, p_g).  
     \label{p_g-p-p_l-est}
 \end{align}
\end{lemma}
\begin{proof}  From (\ref{glrelation-g}) we have  for $p_g \geq p_l$,
\begin{align*}
   p_g^+ \leq |p| - \int_0^{1/2} \frac{\lambda_l(s)}{\lambda(s)}p_c'(s) ds 
                      - \int_{1/2}^1 \frac{\lambda_l(s)}{\lambda(s)}p_c'(s) ds.
\end{align*}
From \Hb{8} it follows that the first integral on the right hand side is bounded and therefore we have
\begin{align*}
   p_g^+ \leq |p| + C + p_c(1/2).
\end{align*}
The same inequality obviously holds also for $p_g < p_l = p$.

From (\ref{glrelation}) we have 
\begin{align*}
  | S_l p_l | \leq |p| + |S_l  \int_{S_l}^1 \frac{\lambda_g(s)}{\lambda(s)}p_c'(s) ds | 
  \leq |p| +  \int_{S_l}^1 \frac{\lambda_g(s)}{\lambda(s)} s | p_c'(s) | ds .
\end{align*}
Due to \Hb{4} the right hand side integral is bounded, which proves the second inequality. 
The third inequality follows in the same way from (\ref{glrelation-g}) and the fact that due to \Hb{4} and \Hb{8}
the function $\hat{P}(S_l)$ is bounded on $(0,1)$.
Finally, the last inequality follows directly from  (\ref{glrelation-g}), (\ref{glrelation}).
This proves the lemma.
\end{proof}

\subsection{Main assumptions}
\label{section-MA}

 \begin{itemize}
     \item[\Hb{1}] The porosity $\Phi$ belongs to $L^{\infty}(\Omega)$, and there exist constants,
	 $\phi_M\geq \phi_m >0$, such that
		 $\phi_m \leq \Phi(x) \leq \phi_M$ a.e. in $\Omega$. The diffusion coefficient $D$ belongs to $L^\infty(\Omega)$, and there 
		 exists a constant $D_0>0$ such that $D(x)\geq D_0$ a.e. in $\Omega$.

     \item[\Hb{2}] The permeability tensor $\Kb$  belongs to $\left(L^{\infty}(\Omega)\right)^{d\times d}$, 
	 and there exist constants $k_M \geq k_m >0$,  such that for almost all $x\in\Omega$ 
	 and all $\boldsymbol{\xi}\in \Rb^d$ it holds:
 \[ k_m |\boldsymbol{\xi}|^2 \leq \Kb(x)\boldsymbol{\xi}\cdot\boldsymbol{\xi}\leq k_M |\boldsymbol{\xi}|^2.\]

 \item[\Hb{3}] Relative mobilities $\lambda_l,\lambda_g$ are defined as $\lambda_l(S_l)= kr_l(S_l)/\mu_l$ 
	 and  $\lambda_g(S_l)= kr_g(S_l)/\mu_g$ where the constants $\mu_l >0$ and $\mu_g >0$
	 are the liquid and the gas viscosities, and $ kr_l(S_l)$, $ kr_g(S_l)$ are the relative permeability
	 functions, satisfying $kr_l,kr_g \in C([0,1])$, $kr_l(0)=0$
      and $kr_g(1)=0$; the function $kr_{l}$ is a non decreasing and $kr_g$ is non increasing function of $S_{l}$.
      Moreover, there exist  constants $kr_m >0$
      such that for all $S_l\in [0,1]$
      \[  kr_m \leq kr_l(S_l) + kr_g(S_l). \]
      We assume also that there exists a constant $a_l > 0$ such that for all $S_l \in [0,1]$:
      \begin{equation}
	      \label{assumpt:kr_l}
	      a_l S_l^2 \leq kr_l(S_l) .
      \end{equation}
 
    \item[\Hb{4}]  The capillary pressure function,  $p_c\in C^1(0,1)\cap C^0( (0,1])$, is strict monotone decreasing function of $S_l\in (0,1]$ satisfying 
      $p_c(1) = 0$,    $p_c(S_l) >0$ for $S_l\in (0,1)$ and ${p}_c'(S_l) \leq -M_0<0 $  for $S_l\in (0,1]$ and some constant $M_0>0$.
      There exists a positive constant $M_{p_c}$  such that
     \begin{align}
		 \int_0^{1}p_c(s)\, ds = M_{p_c} < +\infty. 
     \end{align}
		 The inverse functions $p_c^{-1}$ is extended as $p_c^{-1}(\sigma) = 1$ for $\sigma \leq 0$. 

   \item[\Hb{5}] 
		 The function $\hat u(p_g)$ is strictly increasing $C^1$ function from $[0,+\infty)$
		 to  $[0,+\infty)$ and $\hat{u}(0)=0$.
		 There exist  constants $u_{max} > 0$ and $M_g > 0$  such that
	         for all $\sigma \geq 0$ it holds,
     \begin{align*}
	  | \hat{u} (\sigma) | \leq u_{max},\quad   0 < \hat{u}'(\sigma)\leq M_g.
     \end{align*}
     For $\sigma \leq 0$ we extend $\hat{u}(\sigma)$ as a smooth, sufficiently small, bounded function having global $C^1$ regularity. 
     The main low solubility assumption is that the constant $M_g$ is sufficiently
     small, namely that the inequality (\ref{low_dissolution}) holds.
		 
 \item[\Hb{6}] 
   Function $\hat\rho_g(p_g)$ is a $C^1$ strictly increasing function on $[0,\infty)$, and there exist  constants $\rho_M >0$ and $\rho_g^{max} > 0$ such that for all $p_g\geq 0$ it holds 
     \begin{align*}
	     0 \leq  \hat\rho_g(p_g) \leq \rho_M,  \quad | \hat{\rho}'_g(p_g)| \leq \rho_g^{max}, \quad \hat\rho_g(0) = 0,\quad 
	     \int_0^1 \frac{d\sigma}{\hat{\rho}_g(\sigma) } < \infty.
     \end{align*}
     For $\sigma \leq 0$  we set  $\hat\rho_g(\sigma)=0$ for all $\sigma\leq 0$. 
 \item[\Hb{7}]
  $F_I, F_P \in L^2(Q_T)$  and  $F_I, F_P , p_g^{0} \geq 0$ a.e. in $Q_T$.

 \item[\Hb{8}] The function $\alpha(S_l)$ defined in (\ref{alpha-beta}) satisfy $\alpha\in C^0([0,1])$, $\alpha(0)=\alpha(1) = 0$, and $\alpha(S) >0$ 
	 for $S\in (0,1)$.
	 The inverse of the function $\beta(S_l)$, defined in (\ref{alpha-beta}), is H\"older continuous function of order $\tau \in (0,1)$,
   which can be written as (for some positive constant $C\geq 0$.)
 \begin{align}
 C\left|\int_{S_1}^{S_2}\alpha(s)\, ds\right|^{\tau}\geq |S_1-S_2|.
 \end{align} 
 \item[\Hb{9}] The function $(1-S_l)\hat{P}(S_l)$, where $\hat{P}(S_l)$ is defined in (\ref{glrelation-g}), is H\"older continuous 
   for $S_l\in (0,1)$ with some exponent    $\overline{\tau}\in (0,1]$.
 \end{itemize}


\begin{remark} \label{rem-uro}
Boudnedness of the function $\hat u$ from \Hb{5} is a simplification that is not restrictive since 
$u_{max}$ can take arbitrary large values. The same is true for boundedness of the gas density in \Hb{6}. 
\end{remark}

\begin{remark}
   \label{rem-u}
   The function $\hat u(p_g)$ from \Hb{5} has a physical meaning only for non negative values of the 
   pseudo pressure $p_g$. Regularizations applied in Section~\ref{ss-5.3} destroy minimum principle that 
   enforces $p_g\geq 0$ and therefore we need to extend $\hat u(p_g)$  for negative values of $p_g$ as a smooth function.
   This extension is arbitrary and we take it sufficiently small, such that 
\[
  0 < \rho_l^{std} - {u}_{min}\leq  \rho_l = \rho_l^{std} + \hat u(p_g) \leq \rho_l^{std} + u_{max},
  \]
  for some constant $0< u_{min} <  \rho_l^{std}$ and $u_{min}\leq u_{max}$. For reasons which appear in the proof
  of Lemma~\ref{lemma:diss} we also suppose $u_{min} \leq \rho_l^{std}(1 - 1/\sqrt{2})$.
\end{remark}

\begin{remark}
 Assumptions on H\"older continuity in \Hb{8} and \Hb{9} are needed in the compactness proof in section~\ref{sec:proof-1}.
 Assumption \Hb{8} is usual in the two-phase flow models, while assumption \Hb{9} is fulfilled if $(1-S_l) p_c^\prime(S_l)$
 is an $L^p$ function, for $p>1$, away from $S_l=0$. This assumption is a consequence of \Hb{8} ($\alpha(1) = 0$)
 if, for example,  $kr_g(S_l) \geq C (1-S_l)^\gamma$ for some $0<\gamma < 2$.
\end{remark}
 \begin{lemma}
	 \label{lemma:diss}
	 Let the assumptions  \Hb{1}-\Hb{8} be fulfilled and let $u=\hat{u}(p_g)$ and $\rho_l = \rho_l^{std}+\hat{u}(p_g)$. 
By $z$ we denote the number
\begin{align*}
  z = \min_{0\leq S_l\leq 1} (  kr_g(S_l) + S_l),
\end{align*}
and we suppose that $M_g$ in \Hb{5} is sufficiently small, namely we assume:
 \begin{align}
	 \label{low_dissolution}
	 \frac{\Phi  D}{\rho_l^{std} k_{m}/ \mu_l} 
	 \max\left( \frac{\rho_M}{\rho_l^{std}}\frac{1}{ a_l z},  \sqrt{ \frac{\mu_g}{\mu_l}}\frac{1}{ \sqrt{a_l z}}\right)
	 <  \frac{1}{M_g}. 
\end{align}
	 Then the following inequalities hold: 
\begin{align}
	\label{est:diss-1}
   c_D  |  \nabla u |^2 \leq
   \lambda_g(S_l) \Kb \nabla p_g\cdot  \nabla p_g + S_l \Phi D \frac{1}{\rho_g}\nabla p_g\cdot  \nabla u,\\
	\label{est:diss-1a}
| \frac{1}{\rho_l}\Phi S_l D  \nabla u  \cdot \nabla p_l | \leq 
\frac{1}{2} \lambda_l(S_l) \Kb \nabla p_l\cdot  \nabla p_l + q c_D  |  \nabla u |^2 ,
\end{align}
for some $0< q <1$, where
\begin{align}
	\label{est:diss-2}
    c_D  =  \frac{\Phi^2  D^2 \mu_l}{(\rho_l^{std})^2 k_{m}a_l} .
\end{align}
 \end{lemma}
 \begin{proof} We have 
\begin{align*}
\lambda_g(S_l) \Kb \nabla p_g\cdot  \nabla p_g + S_l \Phi D \frac{1}{\rho_g}\nabla p_g\cdot  \nabla u
\geq  \left(\frac{ 1 }{\mu_g \hat{u}'(p_g)^2}  kr_g(S_l) k_{m}
    + S_l\Phi  D \frac{1}{\rho_g \hat{u}'(p_g)} \right)  |\nabla u|^2,
\end{align*}
from where it follows:
\begin{align*}
\lambda_g(S_l) \Kb \nabla p_g\cdot  \nabla p_g + S_l \Phi D \frac{1}{\rho_g}\nabla p_g\cdot  \nabla u
\geq z \min\left(\frac{  k_{m} }{\mu_g M_g^2}, \frac{  \Phi  D }{\rho_M M_g} \right)
 |\nabla u|^2.
\end{align*}
Estimate (\ref{est:diss-1}) and (\ref{est:diss-2}) follow immediately from  
\begin{align}
 \label{low_dissolution-0}
 \frac{\Phi^2  D^2 \mu_l}{(\rho_l^{std})^2 k_{m}a_l} \leq
z \min\left(\frac{ k_{m} }{\mu_g M_g^2} ,  \frac{\Phi  D}{\rho_M M_g} \right).
\end{align}
It is easy to show that  (\ref{low_dissolution-0})  follows from  (\ref{low_dissolution}) and the
fact that $a_l$ can be taken arbitrary small, such that $a_lz \leq 1$; this proves (\ref{est:diss-1}).

To prove (\ref{est:diss-1a}) we note that since the extension of the function $\hat{u}$ into negative pseudo-pressures
can be taken arbitrary small, we have $\rho_l \geq \rho_l^{std}(1-\varepsilon)$, for $0 <\varepsilon = u_{min}/\rho_l^{std} <1-1/\sqrt{2}$
(see Remark~\ref{rem-u}). Therefore we can estimate
\begin{align*}
	| \frac{1}{\rho_l}\Phi S_l D  \nabla u  \cdot \nabla p_l | & \leq 
	| \frac{1}{\rho_l^{std} (1-\varepsilon)\sqrt{k_{m} \lambda_l(S_l)}}\Phi S_l D \sqrt{k_{m} \lambda_l(S_l)} \hat{u}'(p_g) \nabla p_g \cdot \nabla p_l |\\
   & \leq  \frac{1}{2} k_{m} \lambda_l(S_l) \nabla p_l \cdot \nabla p_l
	+   \frac{1}{2 (1-\varepsilon)^2}\frac{\Phi^2 S_l^2 D^2}{(\rho_l^{std})^2 k_{m} \lambda_l(S_l)} \hat{u}'(p_g)^2 \nabla p_g \cdot \nabla p_g\\
   &\leq \frac{1}{2} \lambda_l(S_l) \Kb \nabla p_l\cdot  \nabla p_l +q c_D |  \nabla u |^2,
\end{align*}
where in the last step we have used (\ref{assumpt:kr_l}), and $q =0.5/(1-\varepsilon)^2 < 1$.
Lemma~\ref{lemma:diss} is proved.
\end{proof}

\begin{remark} 
	Exact meaning of the low solubility hypothesis is given by (\ref{low_dissolution}). 
The {\rm solubility bound} $M_g$ must be small enough so that $1/M_g$ is larger than a ratio of diffusivity $\Phi  D$
and hydraulic conductivity $\rho_l^{std} k_{m}/ \mu_l$ multiplied by generally small
	non-dimensional factor. 
\end{remark}

\begin{remark} 
If we take as an example the  flow of water and hydrogen modeled by the Henry law,
	$\hat u(p_g)= H(T)M^h p_g,$
we can check that the inequality (\ref{low_dissolution}) is realistic. Some typical values for corresponding  parameters
 (at $T=303 K$) are the following: $H(T) = 7.65\cdot 10^{-6}$ $\mathrm{mol/m^2 Pa}$,  $\rho_l^{std} = 10^3$  $\mathrm{kg/ m^3}$,
 $M^h=2\cdot 10^{-3}$  $\mathrm{kg/mol}$, $\mu_l = 10^{-3}$ and $\mu_g = 6\cdot 10^{-6}$  $\mathrm{Pa\cdot s}$, $k_m = 10^{-19}$
 $\mathrm{m^2}$, $\Phi=0.1$, $D = 3\cdot 10^{-9}$   $\mathrm{m^2/s}$, $a_l = 1$ and $z=0.1$. 
 {With these values of the parameters} we get that $1/M_g$ should be bigger than $3\cdot 10^4$, while 
 $M_g=H(T)M^h = 15.3 \cdot 10^{-9}$ (see \cite{AAJ}). 
\end{remark}

\section{Existence theorem}
\label{existence}

Let us recall that the primary variables are $p_l$ and $p_g$. The secondary variables are the global pressure 
$p$ defined by (\ref{glrelation}) 
and the functions $u$, $\rho_g$, $S_l$ and $S_g$ defined as $u=\hat{u}(p_g)$, 
$\rho_g=\hat{\rho}_g(p_g)$, $S_l = p_c^{-1}(p_g - p_l)$ and $S_g = 1-S_l$.
By \Hb{5} and \Hb{6} the functions $u$ and $\rho_g$ are bounded and for $S_l$, due to \Hb{4}, we have 
\begin{align}
	\label{Sl-minmax}
	0< S_l \leq 1.
\end{align}
 
Variational formulation is obtained by standard arguments. Taking test functions $\varphi,\psi\in C^1([0,T], V)$  where 
\[ V=\{\varphi\in H^1(\Omega)\colon \varphi = 0 \text{ on } \Gamma_D\}\]
we get: 

\begin{theorem}\label{TM1}
    Let  \Hb{1}-\Hb{8} hold true
    and assume $(p_{l}^0, p_g^{0})\in L^2(\Omega)\times L^2(\Omega)$, {$p_g^0\geq 0$ a.e.\ in $\Omega$}.   
Then there exist functions $p_l$ and $p_g$ satisfying  
\begin{align*}
	p_l, p_g\in L^2(Q_T),\quad   p, u, \beta(S_l)\in  L^2(0,T;V), \\
  \Phi  \partial_t (u S_l+ \rho_g S_g),\;\Phi  \partial_t S_l \in  L^2(0,T; V'),
   \end{align*}
such that: for all $\varphi\in L^2(0,T; V)$
\begin{equation}
  \label{VF-1}
  \begin{split}
 \int_0^T \langle \Phi\dert{S_l},\varphi\rangle dt 
 & +\int_{Q_T} [ \lambda_l(S_l)\Kb\nabla p_l - \Phi S_l \frac{1}{\rho_l} D  \nabla u]  \cdot \nabla \varphi dx dt\\
 & +\int_{Q_T} S_l F_P \varphi dx dt     = \int_{Q_T} F_I \varphi dx dt   +
     \int_{Q_T} \rho_l\lambda_l(S_l)\Kb \gb \cdot \nabla \varphi dx dt; 
   \end{split}
\end{equation}
	 for all $\psi\in L^2(0,T; V)$  
\begin{equation}
  \label{VF-2}
  \begin{split}
   & \int_0^T \langle \Phi\dert{}\left(   u S_l +   \rho_g S_g\right),\psi\rangle dt\\
   & +\int_{Q_T}[  u \lambda_l(S_l)\Kb  \nabla p_l +  \rho_g  \lambda_g(S_l) \Kb \nabla p_g  
    + \Phi S_l \frac{\rho_l^{std}}{\rho_l} D  \nabla u ]  \cdot \nabla \psi dx dt\\
  & +\int_{Q_T} \left(   u S_l +   \rho_g S_g\right) F_P\psi dx dt 
  =\int_{Q_T}\left(  \rho_l u \lambda_l(S_l)+  \rho_g^2  \lambda_g(S_l) \right) \Kb\gb  \cdot \nabla \psi dx dt.
   \end{split}
\end{equation}
Furthermore, for all $\psi\in V$ the functions 
\[ t\mapsto \int_\Omega \Phi S_l \psi dx ,\quad
t\mapsto \int_\Omega \Phi ((u - \rho_g)S_l+\rho_g) \psi dx 
\]
are continuous in $[0,T]$ and the initial condition is satisfied in the following sense:
\begin{equation} \label{VF-3}
\left( \int_\Omega \Phi S_l \psi dx\right)(0) = 
\int_\Omega \Phi s_0\psi dx,
\end{equation}
\begin{equation} \label{VF-4}
\left( \int_\Omega \Phi (u S_l + \rho_g S_g) \psi dx\right)(0) = 
\int_\Omega \Phi (\hat{u}(p_g^0)s_0+\hat\rho_g(p_g^0)(1-s_0))\psi dx,
\end{equation}
for all $\psi\in V$, where $s_0=p_c^{-1}(p_g^0-p_l^0)$.
\end{theorem}

The first step in proving correctness of the proposed model for two-phase compositional flow is to show that 
the weak solution defined in Theorem~\ref{TM1} satisfy $p_g \geq 0$ a.e.\ in $Q_T$,
if the initial and the boundary conditions satisfy 
corresponding inequality. 

\begin{lemma}
  \label{pozitivity}
		Let $p_l$ and $p_g$ are given by Theorem~\ref{TM1}. Then, $p_g \geq 0$ a.e. in $Q_T$.
\end{lemma}
Lemma~\ref{pozitivity} can be proved by standard technique using  test function
 $\varphi = (\min(\hat{u}(p_g),0))^2/2$ in equation  (\ref{VF-1}) and 
the function $\psi=\min(\hat{u}(p_g),0)$ in equation  (\ref{VF-2}). 
The proof is omitted here since it will be given in the discrete case in Lemma~\ref{lemma-pos}.

The proof of  Theorem~\ref{TM1} is based on an energy estimate obtained by the use of test functions 
\begin{align*}
    \varphi =  p_l-N(p_g) ,\quad \psi = M(p_g),
\end{align*}
with
\begin{equation}\label{test-fun-def}
M(p_g)=\int_0^{p_g^+}\frac{1}{\hat\rho_g(\sigma)}\, d\sigma\quad 
N(p_g)=\int_0^{p_g^+}\frac{\hat{u}(\sigma)}{\hat\rho_g(\sigma)}d\sigma.
\end{equation}
It is assumed that the functions $M$ and $N$ are extended by zero for negative pressures. 
For $M$ and $N$ we have the following bounds:
\begin{lemma}\label{lema-test-fun}
  Functions (\ref{test-fun-def}) satisfy
    \begin{align}
	| N(p_g) | \leq \hat{C}_g u_{max} (p_g^+ +1), \quad
	| M(p_g) | \leq  \hat{C}_g(p_g^+ +1) ,
	\label{M-N-bounds}
\end{align}
where $\hat{C}_g = \max( \int_0^1 d\sigma/\hat{\rho}_g(\sigma), 1/\hat{\rho}_g(1))$. 
   \end{lemma}
\begin{proof}
Due to \Hb{5} and \Hb{6} we have:
\begin{align*}
	| N(p_g) | &\leq \int_0^{1} \frac{\hat{u}(\sigma)}{\hat{\rho}_g(\sigma)}\, d\sigma 
	+ \int_1^{\max(p_g^+,1)} \frac{\hat{u}(\sigma)}{\hat{\rho}_g(\sigma)}\, d\sigma 
	 \leq u_{max} \left( \int_0^{1} \frac{1}{\hat{\rho}_g(\sigma)}\, d\sigma + \frac{1}{\hat\rho_g(1)}p_g^+ \right) \\
	| M(p_g) | &\leq \int_0^{1}\frac{1}{\hat\rho_g(\sigma)}\, d\sigma  
	+\int_1^{\max(p_g^+,1)}\frac{1}{\hat\rho_g(\sigma)}\, d\sigma 
	 \leq \int_0^{1}\frac{1}{\hat\rho_g(\sigma)}\, d\sigma  +\frac{1}{\hat\rho_g(1)} p_g^+.
\end{align*}
Lemma~\ref{lema-test-fun} is proved.
\end{proof}

The key property of the test functions $p_l-N(p_g)$ and $M(p_g)$ is given by the following relation 
\begin{align}
 \dert{S_l} (p_l-N(p_g))  +  \dert{}\left( u S_l +   \rho_g S_g\right)M(p_g)
=  \dert{}{\cal E}(p_l,p_g), \label{TSTF-1}
\end{align}
where the function $\mathcal{E}$ is given by
\begin{equation}
  \label{E-final}
  \begin{split}
    \mathcal{E}(p_l,p_g)&= S_l\left(\hat{u}(p_g) M(p_g) - N(p_g)\right) + S_g\left( \hat{\rho}_g(p_g) M(p_g) - p_g\right)
 - \int_0^{S_l} p_c(s) ds .
  \end{split}
\end{equation}


\begin{lemma}\label{lema-E-est}
	The function ${\cal E}$ defined in (\ref{E-final}) satisfy:
	\begin{align}
	-M_{p_c} \leq {\cal E}(p_l,p_g) \leq C(|p_g| +1).\label{E-bound}
	\end{align}
	for all $p_l\in \mathbb{R}$ and $p_g \geq 0$, 
	where the constant $C$ depends on $u_{max}$, $\rho_{M}$,  $\hat{C}_g$ and $M_{p_c}$.
\end{lemma}
\begin{proof}. Using monotonicity of the gas mass density we have 
	\begin{align*}
	  \hat{u}(p_g) \int_0^{p_g}\frac{1}{\hat \rho_g(\sigma)}d\sigma
	&-\int_0^{p_g}\frac{\hat u(\sigma)}{\hat \rho_g(\sigma)}d\sigma
	\geq  {u \int_0^{p_g}\frac{1}{\hat \rho_g(\sigma)}d\sigma-u\int_0^{p_g}\frac{ 1 }{\hat \rho_g(\sigma)}d\sigma} =0.
	\end{align*}
	By the same argument,   
	\begin{align*}
	S_g\hat \rho_g(p_g)\int_0^{p_g}\frac{1}{\hat \rho_g(\sigma)}d\sigma
	\geq S_g \hat \rho_g(p_g)\cdot \frac{1}{\hat \rho_g(p_g)}\int_0^{p_g}d\sigma
	= S_g p_g .
	\end{align*}
	Therefore, we have  the estimate:
	\begin{align*}
	{\cal E}(p_l,p_g) &\geq  -\int_0^{S_l}p_c(s)\, ds\geq -M_{p_c}  .
	\end{align*}
	The upper bound 
	follows directly from the estimates on the functions $M$ and $N$ in Lemma~\ref{lema-test-fun}.
	Lemma~\ref{lema-E-est} is proved.     
\end{proof}

By the use of above test functions one can formally prove the following a priori estimates: 
\begin{lemma}
\label{lemma-AE1}
	Let the assumptions  \Hb{1}-\Hb{8} be fulfilled and let the initial conditions $p_l^0$ and 
	$p_g^{0}$ be such that ${\cal E}(p_l^0, p_g^{0}) \in L^1(\Omega)$.  Then there is a constant 
	$C$ such that each solution of \eqref{VF-1}, \eqref{VF-2} satisfy:
\begin{align}
  \int_{Q_T}\left\{\lambda_l(S_l)|\nabla p_l|^2+\lambda_g(S_l)|\nabla p_g|^2+|\nabla u |^2\right\}\leq C, \label{bae1}\\
  \int_{Q_T}\left\{|\nabla p|^2+|\nabla \beta(S_l)|^2+|\nabla u|^2\right\}\leq C, \label{bae2}\\
\left\|\partial_t(\Phi [ u S_l+\rho_g S_g ])\right\|_{L^2(0,T; H^{-1}(\Omega))}
+\left\|\partial_t(\Phi S_l)\right\|_{L^2(0,T; H^{-1}(\Omega))}\leq C. \label{bae3}
\end{align}
\end{lemma}
We shall not give a direct proof of Lemma~\ref{lemma-AE1} since it will be proved for regularized 
problem and then inferred by passing to the limit in a regularization parameter.

\section{Regularized $\eta$-problem and time discretization}

The system of equations (\ref{eq-1B}), (\ref{eq-2B}) contains several degeneracies and, 
as a consequence, the 
phase pressures do not belong to $L^2(0,T;H^1(\Omega))$ space; the same is true for the
capillary pressure and the saturation. 
In the first regularisation step we will add some terms into governing equations that will make 
the capillary pressure  $L^2(0,T;H^1(\Omega))$ function. Then, using (\ref{glrelation1-grad}), we may conclude that the
regularized phase 
pressures are also $L^2(0,T;H^1(\Omega))$ functions. The regularized system
is as follows:
\begin{equation}
    \Phi\dert{S_l^{\eta}} + \dv \mathbf{Q}^{w,\eta} +S_l^{\eta} F_p =  F_I, 
 \label{eq-1Breg} 
\end{equation}
\begin{equation}
\Phi\dert{}\left(   u^{\eta} S_l^{\eta} +   \rho_g^{\eta} S_g^{\eta}\right)
+ \dv \mathbf{Q}^{h,\eta}
 +(u^{\eta}S_l^{\eta}+\rho_g^{\eta}S_g^{\eta})F_P= 0,
 \label{eq-2Breg}
\end{equation}
where the fluxes are given by:
\begin{equation}
\label{eq-3Breg} 
\mathbf{Q}^{w,\eta}= -\lambda_l(S_l^{\eta})\Kb(\nabla p_l^{\eta} -\rho_l^{\eta}\gb)+\Phi S_l^{\eta} \frac{1}{\rho_l^{\eta}}
                   D  \nabla u^{\eta}+\eta\nabla (p_g^{\eta}-p_l^{\eta}),
\end{equation}
\begin{equation}
\label{eq-4Breg} 
\begin{split}
\mathbf{Q}^{h,\eta}= & - u^{\eta} \lambda_l(S_l^{\eta})\Kb \left(\nabla p_l^{\eta} - \rho_l^{\eta} \gb\right) 
          - \rho_g^{\eta} \lambda_g(S_l^{\eta})\Kb\left(\nabla p_g^{\eta} -\rho_g^{\eta} \gb\right) \\
          & - \Phi S_l^{\eta} \frac{\rho_l^{std}}{\rho_l^{\eta}} D  \nabla u^{\eta}
		   - \eta (\rho_g^{\eta}-u^{\eta})\nabla (p_g^{\eta}-p_l^{\eta}).
\end{split}
\end{equation}
The system is  completed with the initial  and the boundary conditions:
\begin{equation}
	\label{eq-5Breg}
  \begin{split}
  p_l^{\eta}(x,0)=p_l^{0}(x),\quad p_g^{\eta}(x,0)= p_g^{0}(x)& \qquad   \mathrm{in}\ \Omega\\ 
p_l^{\eta}(x,t)=0,\quad p_g^{\eta}(x,t)=0 & \qquad \mathrm{on}\ (0,T)\times \Gamma_D\\
\mathbf{Q}^{h,\eta}\cdot {\bf n}=0,\quad   
\mathbf{Q}^{w,\eta}\cdot {\bf n}=0 &\qquad    \mathrm{on}\ (0,T)\times \Gamma_N
\end{split}
\end{equation}
The secondary variables used in (\ref{eq-1Breg})--(\ref{eq-4Breg}) are defined as:
\begin{align*}
u^{\eta}=\hat u(p_g^{\eta}),\; \rho_l^{\eta}=\rho_l^{std}+u^{\eta},\; \rho_g^{\eta}=\hat\rho_g(p_g^{\eta}),\;
S_l^{\eta}=p_c^{-1}(p_g^{\eta}-p_l^{\eta}),\; S_g^{\eta}=1-S_l^{\eta}.
\end{align*}

We shall first prove the following theorem which states the existence of a weak solution to the problem 
(\ref{eq-1Breg})--(\ref{eq-5Breg}) and then, by passing to the limit as small parameter $\eta$ tends to zero, 
the existence of a weak solution to degenerated system (\ref{eq-1B})--(\ref{initial}). 

\begin{theorem}\label{TM2}
    Let  \Hb{1}-\Hb{8} hold
    and assume $(p_{l}^0, p_g^{0})\in H^1(\Omega)\times H^1(\Omega)$, $p_g^0\geq 0$.
Then  for all $\eta>0$ there exists $(p_l^{\eta},p_g^{\eta})$ with $p_g^{\eta}\geq 0$ a.e. in $Q_T$, satisfying  
\begin{align*}
p_l^{\eta},  p_g^{\eta}, u^\eta  \in L^2(0,T;V),\; \\
  \Phi  \partial_t (u^{\eta} S_l^{\eta}+\rho_g^{\eta}S_g^{\eta}),
   \Phi  \partial_t(S_l^{\eta} )\in  L^2(0,T; V'),\\
	u^{\eta} S_l^{\eta}+\rho_g^{\eta}S_g^{\eta} \in C^{0}([0,T];L^{2}(\Omega)),  S_l^{\eta}\in C^{0}([0,T]; L^{2}(\Omega)).
   \end{align*}
For all $\varphi \in L^2(0,T; V)$,
\begin{equation}\label{TM2-ve1}
\begin{split}
     \int_0^T \langle \Phi\dert{S_l^{\eta}},\varphi\rangle dt 
     +\int_{Q_T} [\lambda_l(S_l^{\eta})\Kb\nabla p_l^{\eta} - \Phi S_l^\eta \frac{1}{\rho_l^{\eta}}D\nabla u^{\eta}-\eta\nabla(p_g^{\eta}-p_l^{\eta})]  \cdot \nabla \varphi dx dt\\
    +\int_{Q_T}S_l^{\eta}F_p\varphi\, dx dt =\int_{Q_T}F_I \varphi\, dx dt   +
     \int_{Q_T}\rho_l^\eta\lambda_l(S_l^\eta)\Kb \gb \cdot \nabla \varphi dx dt. 
\end{split}
\end{equation}
For  all $\psi \in  L^2(0,T; V)$,   
\begin{equation}\label{TM2-ve2}
\begin{split}
   & \int_0^T \langle \Phi\dert{}\left(   u^{\eta} S_l^{\eta} +   \rho_g^{\eta} S_g^{\eta}\right),\psi\rangle dt\\
 +  & \int_{Q_T}[ u^{\eta} \lambda_l(S_l^{\eta})\Kb  \nabla p_l^{\eta} +  \rho_g^{\eta}  \lambda_g(S_l^{\eta}) \Kb \nabla p_g^{\eta}  
 + \Phi S_l^{\eta}\frac{\rho_l^{std}}{\rho_l^{\eta}}D\nabla u^{\eta}] \cdot \nabla \psi dx dt\\
 +  &\eta  \int_{Q_T}(\rho_g^{\eta}-u^{\eta})\nabla (p_g^{\eta}-p_l^{\eta}) \cdot \nabla \psi dx dt\\
 + & \int_{Q_T} (u^\eta S_l^\eta +\rho_g^\eta S_g^{\eta})F_P\psi\, dx dt=\int_{Q_T}\left(  \rho_l^{\eta} u^{\eta} \lambda_l(S_l^{\eta})+  (\rho_g^{\eta})^2  \lambda_g(S_l^{\eta}) \right) \Kb\gb  \cdot \nabla \psi dx dt.
\end{split}
\end{equation}
Furthermore, $u^{\eta}S_l^{\eta}+\rho_g^{\eta}(1-S_l^{\eta})=\hat u(p_g^{0}) S_l^0+\hat\rho_g(p_g^{0})(1-S_l^0)$  a.e in $\Omega$ for $t=0,$ and \linebreak $S_l^\eta(x,0)=S_l^0$ a.e in $\Omega$.
\end{theorem}

In the proof of   Theorem~\ref{TM2} we will first discretize the time derivative, reducing 
the problem to a sequence of elliptic problems, which will be solved by an application of the Schauder 
fixed point theorem. 
In order to simplify the notation we will omit in writing the 
dependence  on the small parameter $\eta$ until the passage to the limit as $\eta \to 0$.

The time derivative is discretized in the following way:
 For each positive integer {$M$} we divide $[0,T]$ into {$M$} subintervals of equal length 
 $\delta t= T/M$. We set  $t_n = n \delta t$ and $J_n = (t_{n-1},t_n]$ for {$1\leq n\leq M$}, and
we  denote the time difference operator by
 \[ \partial^{\delta t} v(t) = \frac{v(t+\delta t)-v(t)}{\delta t}.\]
For any Hilbert space ${\cal H}$ we denote
 \[ l_{\delta t}({\cal H}) =\{ v\in L^\infty(0,T;{\cal H})\colon v\text{ is constant in time on each subinterval } 
 J_n \subset [0,T]\}.
 \]
 For $v^{\delta t} \in  l_{\delta t}({\cal H})$ we set $v^n = v^{\delta t}|_{J_n}$ and, therefore, we can write
 \begin{align*}
     v^{\delta t} = {\sum_{n=1}^{M}} v^{n} \chi_{(t_{n-1},t_n]}(t),\quad v^{\delta t}(0) = v^0.
 \end{align*}
 To function $v^{\delta t} \in  l_{\delta t}({\cal H})$ one can assign a piecewise linear in time function 
 \begin{align}
     \tilde{v}^{\delta t} = {\sum_{n=1}^{M}} \left(  \frac{t_n-t}{\delta t} v^{n-1} +  \frac{t-t_{n-1}}{\delta t} v^{n}\right) \chi_{(t_{n-1},t_n]}(t),
     \quad \tilde{v}^{\delta t}(0) = v^0. \label{tilda-operator}
 \end{align}
Then we have
    $\partial_t \tilde{v}^{\delta t}(t) = \partial^{-\delta t} v^{\delta t}(t)$, for $t\neq n\delta t, n=0,1,\ldots, N$.
 Finally, for any function
 $f\in L^1(0,T; {\cal H})$ we  define $f^{\delta t} \in  l_{\delta t}({\cal H})$ by,
 \[ f^{\delta t}(t) = \frac{1}{\delta t} \int_{J_n} f(\tau) d\tau, \quad t\in J_n.\]
The discrete secondary variables are denoted as before by:
\begin{align*}
u^{\delta t}=\hat u(p_g^{\delta t}),\;
\rho_l^{\delta t}=\rho_l^{std} + u^{\delta t},\;
\rho_g^{\delta t}=\hat \rho_g(p_g^{\delta t}),\;
S_l^{\delta t}=p_c^{-1}(p_g^{\delta t}-p_l^{\delta t}).
\end{align*}
 with  the boundary conditions (\ref{eq-5Breg}) and the initial values of the phase pressures which are given by
The weak formulation of the discrete in time system is as follows: 
For given $p_l^0$ and  $p_g^{0}$ find $p_l^{\delta t} \in l_{\delta t}(V)$ and $p_g^{\delta t}\in l_{\delta t}(V)$ satisfying
 \begin{align}\nonumber
 \int_{Q_T} \Phi \partial^{-\delta t} (S_l^{\delta t})\varphi \, dx\, dt 
  &+\int_{Q_T} [\lambda_l(S_l^{\delta t})\Kb\nabla p_l^{\delta t} - \Phi S_l^{\delta t} \frac{1}{\rho_l^{\delta t}}D\nabla u^{\delta t}]\cdot \nabla \varphi \, dx\, dt\\ \label{discrete-1dt}
      &-\eta\int_{Q_T}[\nabla( p_g^{\delta t}-p_l^{\delta t})]  \cdot \nabla \varphi\,  dx\,  dt + \int_{Q_T}S_l^{\delta t}F_P^{\delta t}\varphi \, dx\, dt\\ \nonumber 
      &= \int_{Q_T} F_I^{\delta t}\varphi \, dx\,  dt   +
      \int_{Q_T}\rho_l^{\delta t}\lambda_l(S_l^{\delta t})\Kb \gb \cdot \nabla \varphi\,  dx\,  dt
\end{align}
for all $\varphi\in l_{\delta t}(V)$;
\begin{align}\nonumber
  \int_{Q_T} &\Phi\partial^{-\delta t}\left(  u^{\delta t} S_l^{\delta t} +   \rho_g^{\delta t}(1-S_l^{\delta t})\right)\, \psi\, dx\,  dt\\ \nonumber
    & +\int_{Q_T}\left(u^{\delta t} \lambda_l(S_l^{\delta t})\Kb  \nabla p_l^{\delta t} +  \rho_g^{\delta t}  \lambda_g(S_l^{\delta t}) \Kb \nabla p_g^{\delta t}\right)\cdot \nabla \psi \, dx\, dt\\  \label{discrete-2dt}
     &+\int_{Q_T}\left( \Phi S_l^{\delta t}\frac{\rho_l^{std}}{\rho_l^{\delta t}}D\nabla u^{\delta t}\right)\cdot \nabla \psi \, dx\, dt + \eta\int_{Q_T}\left( (\rho_g^{\delta t}-u^{\delta t})\nabla ( p_g^{\delta t}-p_l^{\delta t})\right)  \cdot \nabla \psi\, dx\,  dt\\ \nonumber
     & + \int_{Q_T} (u^{\delta} S_l^{\delta t}+\rho_g^{\delta t} S_g^{\delta t})F_P^{\delta t}\psi \, dx\, dt\\ \nonumber
   & =\int_{Q_T}\left(  \rho_l^{\delta t}u^{\delta t} \lambda_l(S_l^{\delta t})+  (\rho_g^{\delta t})^2  \lambda_g(S_l^{\delta t}) \right) \Kb\gb  \cdot \nabla \psi dx dt
\end{align}
for all $\psi\in l_{\delta t}(V)$. For $t\leq 0$ we set  $p_l^{\delta t}=p_l^0$, $p_g^{\delta t} = p_g^{0}$.

We will prove the following Theorem~\ref{TM_3} and then, by passing to the limit as 
$\delta t\to 0$, we will  establish Theorem~\ref{TM2}.
\begin{theorem}\label{TM_3}
    Assume  \Hb{1}--\Hb{8}, 
   { $p_g^{0}, p_l^0\in L^2(\Omega)$} and $p_g^0 \geq 0$. Then for all $\delta t$  
   there exist  functions {$p_l^{\delta t}, p_g^{\delta t} \in l_{\delta t}(V)$}, $p_g^{\delta t} \geq 0\text{ a.e. in } Q_T$,
    satisfying  (\ref{discrete-1dt}), (\ref{discrete-2dt}).
\end{theorem}

 The solution of the problem   (\ref{discrete-1dt}), (\ref{discrete-2dt}) is build from a
  sequence of  elliptic problems that we write here explicitly for reader convenience. 
 Let us fix {$1\leq k\leq M$}. 
 We need to establish the existence of functions $p_l^k, p_g^k \in  V$ that satisfy
 \begin{align}\nonumber
  \frac{1}{\delta t}\int_{\Omega}\Phi  (S_l^{k}-S_l^{k-1})\varphi \, dx
       &+\int_{\Omega} [ \lambda_l(S_l^{k})\Kb\nabla p_l^{k} - \Phi S_l^{k} \frac{1}{\rho_l^k}D\nabla u^{k}]\cdot \nabla \varphi \, dx\\ 
	   \label{tmp_p1-1}
       &-\eta\int_{\Omega}[ \nabla p_g^k -  \nabla p_l^{k}]  \cdot \nabla \varphi\,  dx + \int_{\Omega} S_l^k F_P^k \varphi\, dx \\ \nonumber 
       &= \int_{\Omega} F_I^k \varphi \, dx  +
       \int_{\Omega}\rho_l^k\lambda_l(S_l^{k})\Kb \gb \cdot \nabla \varphi\,  dx
 \end{align}
for all $\varphi\in V$ and
\begin{align}\nonumber
  \frac{1}{\delta t}\int_{\Omega} &\Phi \left(\left(u^{k} S_l^{k} +   \rho_g^k(1-S_l^{k})\right)- \left(   u^{k-1} S_l^{k-1} +   \rho_g^{k-1}(1-S_l^{k-1})\right)\right) \, \psi\, dx \\ \nonumber
    & +\int_{\Omega}\left(  u^{k} \lambda_l(S_l^{k})\Kb  \nabla p_l^{k} +  \rho_g^k  \lambda_g(S_l^{k}) \Kb \nabla p_g^k\right)\cdot \nabla \psi \, dx\\  \label{tmp_p1-2}
     &+\int_{\Omega}\left( \Phi S_l^{k}\frac{\rho_l^{std}}{\rho_l^k}D\nabla u^{k}\right)\cdot \nabla \psi \, dx\, dt\\ \nonumber
     & + \eta\int_{\Omega}\left(\rho_g^k - u^k\right)\left( \nabla (p_g^k-p_l^{k})\right)  \cdot \nabla \psi\, dx + \int_{\Omega} (u^k S_l^k+\rho_g^kS_g^k)F_P^{k}\psi\, dx\\ \nonumber
   & =\int_{\Omega}\left(  \rho_l^k u^{k} \lambda_l(S_l^{k})+  (\rho_g^k)^2  \lambda_g(S_l^{k}) \right) \Kb\gb  \cdot \nabla \psi dx
\end{align}
for all $\psi\in V$. 
Here, as always we use notation:
\begin{align*}
u^k=\hat u(p_g^k),\;
\rho_l^k= \rho_l^{std} + \hat u(p_g^k),\;
\rho_g^k=\hat \rho_g(p_g^k),\;
S_l^k=p_c^{-1}(p_g^k-p_l^k).
\end{align*}

\section{Application of the Schauder fixed point theorem}
\label{ss-5.3}

In this section we prove the Theorem~\ref{TM_3}  by proving the 
existence of at least one solution to the problem (\ref{tmp_p1-1}), (\ref{tmp_p1-2}). 
The existence of the solution $(p_l^{k}, p_g^{k})$
for the system  (\ref{tmp_p1-1})--(\ref{tmp_p1-2})   will be proved by
Leray-Schauder's fixed point theorem. This technique is common and is used in \cite{AJPP}, \cite{Khalil-Saad} and similar papers. 
We cite the Leray-Schauder's theorem formulation from \cite{AJPP}.
\begin{theorem}
	Let ${\cal T}$ be a continuous and compact map of a Banach ${\cal B}$ space into itself. 
	Suppose that a set of $x\in {\cal B}$ such that $x=\sigma{\cal T} x$ is bounded for some $\sigma \in [0,1]$. 
	Then the map ${\cal T}$ has a fixed point.
\end{theorem}
In the construction of the
fixed point  map ${\cal T}$ we use  several regularisations.
First, we introduce a small parameter $\varepsilon>0$ and replace $\lambda_l(S_l)$ and $\hat{\rho}_g(p_g)$ 
by 
$$\lambda_l^{\varepsilon}(S_l)=\lambda_l(S_l)+\varepsilon, \quad
\hat{\rho}_g^\varepsilon(p_g)  = \hat{\rho}_g(p_g) + \varepsilon.
$$ 
Function $\lambda_g(S_l)$ is implicitly regularized with the parameter $\varepsilon$ by addition of a new term 
in the equation for the gas phase (see (\ref{tmp_p1-2b})).

Finally, we use  operator
${\EuScript P}_N$ defined as an orthogonal projector in $L^2(\Omega)$
on the first $N$ eigenvectors of the eigenproblem (see \cite{Khalil-Saad}):
\begin{align}
-\Delta p_i = \lambda_i p_i\quad \text{in} \quad \Omega; \nonumber\\
p_i = 0 \quad \text{on } \quad \Gamma_D;
\label{eigen}\\
\nabla p_i\cdot {\bf n} = 0  \quad \text{on} \quad \Gamma_N, \nonumber
\end{align}
and replace several functions by its projections. 
  
It is easy to verify that the operator ${\EuScript P}_N$ satisfy the following properties:
\begin{itemize} 
 	\item[\Pb{1}]
	There exists a constant $C_N$ such that for all $p\in L^2(\Omega)$ and  $q\in V$ it holds
	\begin{equation*}
	\Vert \nabla {\EuScript P}_N \left[ p\right] \Vert_{L^2(\Omega)}
	\leq C_N \, \Vert  p \Vert_{L^2(\Omega)},  \quad \Vert \nabla {\EuScript P}_N \left[ q\right] \Vert_{L^2(\Omega)} \leq \, \Vert  \nabla q \Vert_{L^2(\Omega)}.
	\end{equation*}
	\item[\Pb{2}]  For all $ p \in V $ we have 
	\begin{equation*}
	\int_ {\Omega} \nabla {\EuScript P}_N \left[ p\right] \cdot \nabla p \, dx = \int_ {\Omega} \lvert \nabla {\EuScript P}_N \left[ p\right] \lvert ^ 2 \, dx. 
	\end{equation*}
	\item[\Pb{3}] For $ p, \varphi \in L^2(\Omega) $ we have
	\begin{equation*}
		\int_{\Omega} {\EuScript P}_N \left[ p\right] \varphi \, dx  = \int_{\Omega} p {\EuScript P}_N \left[ \varphi \right] \, dx.
	\end{equation*}
\end{itemize}

From now on, in order to simplify the notation we will omit the superscript $k$ in 
 (\ref{tmp_p1-1}), (\ref{tmp_p1-2}),  and  assume  $k$, $\delta t$ and $\eta$ being fixed.
 All quantities on preceding time level will be denoted by a star ($u^{k-1}$ replaced by
 $u^*$ etc.).  
In order to simplify further notation we will denote function $S_l$ by $S$ in the rest of the section. 

Let  $p_l^*$ and $p_g^*$ be given functions from $L^2(\Omega)$. We define secondary 
variables as:
$$
S^*=p_c^{-1}(p_g^*-p_l^*), \; u^* = \hat{u}(p_g^*),\; 
 \rho_l^* = \rho_l^{std} + u^* ,\; \rho_g^{\varepsilon,*} = \hat{\rho}_g^{\varepsilon}(p_g^*).
 $$
We define  mapping ${\cal T}\colon L^2(\Omega)\times L^2(\Omega)\to L^2(\Omega)\times L^2(\Omega)$ by 
${\cal T}(\overline{p}_l, \overline{p}_g) = (p_l,p_g)$, where $(p_l,p_g)$ 
is a unique solution of linear system (\ref{tmp_p1-1b})--(\ref{tmp_p1-2b}) below. In this system we use 
the following notations:
\begin{align*}
  \overline{S}=p_c^{-1}(\overline{p}_g-\overline{p}_l),\quad 
  \overline{u} = \hat{u}(\overline{p}_g), \quad
  \overline{\rho}_g^\varepsilon = \hat{\rho}_g^\varepsilon(\overline{p}_g),  \quad 
  \overline{\rho}_l = \rho_l^{std} + \overline{u}.
\end{align*}
We also set $\tilde{p}_g = {\EuScript P}_N[ p_g ]$ and consequently  $\tilde{\overline{p}}_g = {\EuScript P}_N[ \overline{p}_g ]$
which leads to the following shorthand notation:
\[
\tilde{\overline{u}} = \hat{u}(\tilde{\overline{p}}_g) ,\quad 
\tilde{\overline{\rho}}_g^\varepsilon = \hat{\rho}_g^\varepsilon(\tilde{\overline{p}}_g) ,\quad 
\tilde{\overline{\rho}}_l = \rho_l^{std} + \tilde{\overline{u}}.
\]
With this notation the 
linearised and regularized variational problem that define the mapping  ${\cal T}$ is given by the following set of equations:
\begin{align}\nonumber
  \frac{1}{\delta t}\int_{\Omega}\Phi &  (\overline{S}-S^*)\varphi \, dx 
  +\int_{\Omega} [ \lambda_l^{\varepsilon}(\overline{S})\Kb\nabla p_l
    - \Phi \overline{S} \frac{1}{\tilde{\overline{\rho}}_l}D\nabla \tilde{\overline{u}}] \cdot \nabla \varphi \, dx\\ 
	\label{tmp_p1-1b}
	&-\eta\int_{\Omega}[ \nabla\tilde{\overline{p}}_g -  \nabla   \tilde{\overline{p}}_l]  \cdot \nabla \varphi\,  dx
       + \int_{\Omega} \overline{S} F_P \varphi\, dx \\ \nonumber 
       &= \int_{\Omega} F_I \varphi \, dx  +
       \int_{\Omega}\overline{\rho}_l\lambda_l(\overline{S})\Kb \gb \cdot \nabla \varphi\,  dx
 \end{align}
for all $\varphi\in V$ and
\begin{align}\nonumber
  \frac{1}{\delta t}&\int_{\Omega} \Phi \left(\left( \overline{u}  \overline{S} +   \overline{\rho}_g^\varepsilon(1- \overline{S})\right)
  - \left(   u^* S^{*} +   {\rho}_g^{\varepsilon, *} (1-S^{*})\right)\right) \, \psi\, dx \\
    & +\int_{\Omega}\left(
    \tilde{\overline{u}} \lambda_l^{\varepsilon}(\overline{S})\Kb  \nabla p_l + \tilde{\overline{\rho}}_g^\varepsilon \lambda_g(\overline{S}) \Kb \nabla \tilde{\overline{p}}_g 
    + \varepsilon\tilde{\overline{\rho}}_g^\varepsilon \nabla p_g +\Phi \overline{S}\frac{\rho_l^{std}}{\tilde{\overline{\rho}}_l } D\nabla  \tilde{\overline{u}} 
    \right)
    \cdot \nabla \psi \, dx \label{tmp_p1-2b}\\
     \nonumber
     & + \eta\int_{\Omega}\left(\tilde{\overline{\rho}}_g^\varepsilon -  \tilde{\overline{u}}\right)
 \left( \nabla   \tilde{\overline{p}}_g -\nabla \tilde{\overline{p}}_l\right)  \cdot \nabla \psi\, dx
 + \int_{\Omega} (\overline{u} \overline{S} +\overline{\rho}_g^\varepsilon (1-\overline{S})) F_P\psi\, dx \nonumber\\ \nonumber
      & =\int_{\Omega}\left(  \overline{\rho}_l \tilde{\overline{u}} \lambda_l(\overline{S})
	  +  (\tilde{\overline{\rho}}_g^\varepsilon)^2  \lambda_g(\overline{S}) \right) \Kb\gb  \cdot \nabla \psi dx
 \end{align}
for all $\psi \in V$. We note that the equations  (\ref{tmp_p1-1b}) and (\ref{tmp_p1-2b})
are linear and uncoupled. Different terms in these equations  are carefully linearised 
in order to keep the symmetry present in original equations that allows us to use the test functions given bellow by (\ref{N-M-delta})
and the orthogonality \Pb{2}.

First we will show that mapping ${\cal T}$ is well defined. Note that (\ref{tmp_p1-1b}) is a 
linear elliptic problem for the function $p_l$, which can be written as
$A_1(p_l, \varphi) = f_1(\varphi)$
with
\begin{align*}
  A_1(p_l, \varphi) &=  \int_{\Omega} \lambda_l^{\varepsilon} (\overline{S}_l)\Kb\nabla p_l \cdot \nabla \varphi \, dx,
\end{align*}
where the functional $f_1(\varphi)$ is given by the remaining terms in the equation (\ref{tmp_p1-1b}).
Using  boundedness of the functions $\hat{u}$ and $\hat{\rho}_g^\varepsilon$, 
$\tilde{\overline{\rho}}_l \geq \rho_l^{std} - u_{min} >0$
and estimates $\| \nabla \tilde{\overline{\xi}}\|_{L^2(\Omega)} \leq C_N \|\overline{\xi}\|_{L^2(\Omega)}$
for $\xi = p_l, p_g, u$, one can easily prove  boundedness of linear functional $f_1$:
\begin{align*}
|f_1(\varphi)|\leq C \lVert \varphi\rVert_{V},
\end{align*}
for all $\varphi\in V$. By the Lax-Milgram lemma, the equation (\ref{tmp_p1-1b}) has a unique solution $p_l\in V$.  

Similarly, since $p_l$ is known from (\ref{tmp_p1-1b}), the equation (\ref{tmp_p1-2b}) can be written as $A_2(p_g, \psi) = f_2(\psi)$
with
\begin{align}
A_2(p_g,\psi) = \int_{\Omega} \varepsilon\tilde{\overline{\rho}}_g^{\varepsilon}\nabla p_g \cdot \nabla \psi \, dx,
\label{Lax-Milgram2}
\end{align}
where linear functional $f_2(\psi)$ is given by the remaining terms in equation (\ref{tmp_p1-2b}).
Using the same arguments as in estimate for $f_1$  we get the boundedness of $f_2$ 
and  by the Lax-Milgram lemma existence of a unique solution $p_g\in V$ to (\ref{tmp_p1-2b}).
This ensures that the map ${\cal T}$ is well defined on $L^2(\Omega)\times L^2(\Omega)$.

\textbf{Continuity and compactness.} 
Let $(\overline p_{l,n}, \overline{p}_{g,n})$ be a sequence in $(L^{2}(\Omega))^2$ 
that converges to some $(\overline{p}_l, \overline{p}_g)$ in $(L^{2}(\Omega))^2$. 
Then we can find a subsequence such that $(p_{l,n}, p_{g,n})={\cal T}(\overline{p}_{l,n},\overline p_{g,n})$
converges weakly in $H^{1}(\Omega)^2$ to some functions  $(p_{l}, p_{g})$. Using continuity and boundedness off all
the coefficients in (\ref{tmp_p1-1b}), (\ref{tmp_p1-2b}), and continuity of the operator ${\EuScript P}_N$, one can easily prove that 
$(p_l,p_g) = {\cal T} (\overline{p}_l, \overline{p}_g)$. The uniqueness of the solution to 
(\ref{tmp_p1-1b}), (\ref{tmp_p1-2b}) gives the convergence of the whole sequence. This proves the 
continuity of the map ${\cal T}$;
the compactness  follows from the compact embedding of $H^1(\Omega) $ into $ L^2(\Omega)$.

\textbf{A priori estimate.}
 Assume that for chosen  $\sigma\in (0,1]$ there exists a pair $(p_l,p_g)$ satisfying 
$(p_l,p_g)=\sigma{\cal T}(p_l,p_g)$, which can be written as:
\begin{align}\nonumber
  \frac{1}{\delta t}\int_{\Omega}\Phi &  ({S}-S^*)\varphi \, dx 
  +\int_{\Omega} [ \lambda_l^{\varepsilon}({S})\Kb\nabla \frac{p_l}{\sigma}
    - \Phi {S} \frac{1}{\tilde{{\rho}}_l}D\nabla \tilde{{u}} ]\cdot \nabla \varphi \, dx\\ \label{eqsystem1}
       &-\eta\int_{\Omega}[ \nabla\tilde{p}_g -  \nabla\tilde{p}_l ]  \cdot \nabla \varphi\,  dx
       + \int_{\Omega} {S} F_P \varphi\, dx \\ \nonumber 
       &= \int_{\Omega} F_I \varphi \, dx  +
       \int_{\Omega}{\rho}_l\lambda_l({S})\Kb \gb \cdot \nabla \varphi\,  dx
 \end{align}
for all $\varphi\in V$,  and
\begin{align}\nonumber
  \frac{1}{\delta t}&\int_{\Omega} \Phi \left(\left( {u}  {S} +   {\rho}_g^{\varepsilon} (1- {S})\right) - 
  \left(   u^* S^{*} +   \rho_g^{\varepsilon, *}(1-S^{*})\right)\right) \, \psi\, dx \\
    & +\int_{\Omega}\left(
    \tilde{u} \lambda_l^{\varepsilon}(S)\Kb  \nabla \frac{p_l}{\sigma}
    + \tilde{\rho}_g^\varepsilon \lambda_g({S}) \Kb \nabla \tilde{p}_g 
    + \varepsilon\tilde{\rho}_g^\varepsilon \nabla \frac{p_g}{\sigma} 
    +\Phi {S}\frac{\rho_l^{std}}{\tilde{\rho}_l } D\nabla  \tilde{{u}} 
    \right)
    \cdot \nabla \psi \, dx \label{eqsystem2}\\
     \nonumber
     & + \eta\int_{\Omega}\left(\tilde{\rho}_g^\varepsilon -  \tilde{u}\right)
	 \left( \nabla\tilde{p}_g - \nabla\tilde{p}_l\right)  \cdot \nabla \psi\, dx
      + \int_{\Omega} (u {S} +\rho_g^\varepsilon (1-S)) F_P\psi\, dx \nonumber\\ \nonumber
   & =\int_{\Omega}\left(  {\rho}_l \tilde{u} \lambda_l(S)
   +  (\tilde{\rho}_g^\varepsilon)^2  \lambda_g({S}) \right) \Kb\gb  \cdot \nabla \psi dx
 \end{align}
for all $\psi \in V$.
Note that in system (\ref{eqsystem1})--(\ref{eqsystem2}) we have two kinds of secondary quantities:
\[ u = \hat{u}(p_g),\quad \rho_g^\varepsilon = \hat{\rho}_g^\varepsilon(p_g),\quad \rho_l = \rho_l^{std} + u,\]
and
\[ \tilde{u} = \hat{u}( \tilde{p}_g) ,\quad  
\tilde{\rho}_g^\varepsilon = \hat{\rho}_g^\varepsilon( \tilde{p}_g),\quad  
\tilde{\rho}_l = \rho_l^{std} +  \tilde{u} .\]
We will use the test functions $\varphi = p_l - N^\varepsilon(\tilde{p}_g)$ and $\psi = M^\varepsilon(\tilde{p}_g)$
given by
\begin{equation}
 N^\varepsilon(p_g) = \int_0^{p_g} \frac{\hat{u}(\sigma)}{\hat{\rho}_g^\varepsilon(\sigma)}d\sigma, \quad
  M^\varepsilon(p_g) = \int_0^{p_g} \frac{1}{\hat{\rho}_g^\varepsilon(\sigma)}d\sigma .
	\label{N-M-delta}
\end{equation}
For any $p_g\in\mathbb{R}$ they  satisfy $\varepsilon$ dependent bounds:
\begin{align}\label{M-N-bounds-neg}
| N^{\varepsilon}(p_g) | \leq \frac{u_{max}}{\varepsilon} |p_g|, \quad
| M^{\varepsilon}(p_g) | \leq \frac{1}{\varepsilon}|p_g|. 
\end{align}
We get
\begin{align*}
  \frac{1}{\delta t}\int_{\Omega}\Phi &  ({S}-S^*)  (p_l - N^\varepsilon(\tilde{p}_g)) \, dx 
  +\int_{\Omega} [ \lambda_l^{\varepsilon}({S})\Kb\nabla \frac{p_l}{\sigma}
    - \Phi {S} \frac{1}{\tilde{{\rho}}_l}D\nabla \tilde{{u}} ]\cdot (\nabla p_l -  
    \frac{\tilde{u}}{\tilde{\rho}_g^\varepsilon}\nabla \tilde{p}_g) \, dx\\ 
       &-\eta\int_{\Omega}[ \nabla \tilde{p}_g -  \nabla\tilde{p}_l]
       \cdot (\nabla p_l -  \frac{\tilde{u}}{\tilde{\rho}_g^\varepsilon}\nabla \tilde{p}_g)\,  dx
       + \int_{\Omega} {S} F_P (p_l - N^\varepsilon(\tilde{p}_g))\, dx \\ 
       &= \int_{\Omega} F_I  (p_l - N^\varepsilon(\tilde{p}_g)) \, dx  +
       \int_{\Omega}{\rho}_l\lambda_l({S})\Kb \gb \cdot (\nabla p_l - 
       \frac{\tilde{u}}{\tilde{\rho}_g^\varepsilon}\nabla \tilde{p}_g)\,  dx,
 \end{align*}
and
\begin{align*}
  \frac{1}{\delta t}&\int_{\Omega} \Phi \left(\left( {u}  {S} +   {\rho}_g^{\varepsilon}  (1- {S})\right) -
  \left(   u^* S^{*} +   \rho_g^{\varepsilon, *} (1-S^{*})\right)\right) \,  M^\varepsilon(\tilde{p}_g)\, dx \\
    & +\int_{\Omega}\left(    \tilde{{u}} \lambda_l^{\varepsilon}({S})\Kb  \nabla \frac{p_l}{\sigma}
    + \tilde{\rho}_g^\varepsilon \lambda_g({S}) \Kb \nabla \tilde{{p}}_g 
    + \varepsilon{\tilde{\rho}}_g^\varepsilon \nabla \frac{p_g}{\sigma} 
    +\Phi {S}\frac{\rho_l^{std}}{\tilde{{\rho}}_l } D\nabla  \tilde{{u}}     \right)   
    \cdot \frac{1}{\tilde{\rho}_g^\varepsilon}\nabla \tilde{p}_g \, dx \\
     & + \eta\int_{\Omega}\left(\tilde{{\rho}}_g^\varepsilon -  \tilde{{u}}\right)
     \left( \nabla \tilde{p}_g -\nabla \tilde{p}_l \right)  \cdot \frac{1}{\tilde{\rho}_g^\varepsilon}\nabla \tilde{p}_g\, dx
      + \int_{\Omega} ({u} {S} +\rho_g^\varepsilon (1-{S})) F_P M^\varepsilon(\tilde{p}_g)\, dx \\ 
   & =\int_{\Omega}\left(  {\rho}_l \tilde{u} \lambda_l({S})+  (\tilde{\rho}_g^\varepsilon)^2  \lambda_g({S}) \right) \Kb\gb  \cdot \frac{1}{\tilde{\rho}_g^\varepsilon}\nabla \tilde{p}_g\, dx.
 \end{align*}
After summation we get (cancellation of four terms and summation of two terms):
\begin{align*}
  &\int_{\Omega} [  \frac{1}{\sigma}\lambda_l^{\varepsilon}({S})\Kb\nabla p_l\cdot \nabla p_l
    + \frac{\varepsilon}{\sigma} \nabla p_g\cdot \nabla \tilde{p}_g
]\, dx\\ 
 &\int_{\Omega} [  
    - \Phi {S} \frac{1}{\tilde{{\rho}}_l}D\nabla \tilde{{u}}\cdot \nabla p_l 
    +  \Phi {S} \frac{1}{\tilde{\rho}_g^\varepsilon} D\nabla \tilde{{u}} \cdot  \nabla \tilde{p}_g 
    + \lambda_g({S}) \Kb \nabla \tilde{{p}}_g \cdot \nabla \tilde{p}_g
]\, dx\\ 
       &+\eta\int_{\Omega}[ \nabla\tilde{p}_g -  \nabla \tilde{p}_l]  
       \cdot ( \nabla \tilde{p}_g - \nabla p_l )\,  dx
       = - \int_{\Omega} {S} F_P (p_l - N^\varepsilon(\tilde{p}_g))\, dx \\ 
       &+ \int_{\Omega} F_I  (p_l - N^\varepsilon(\tilde{p}_g)) \, dx  +
       \int_{\Omega}{\rho}_l\lambda_l({S})\Kb \gb \cdot \nabla p_l \,  dx
       -   \frac{1}{\delta t}\int_{\Omega}\Phi  ({S}-S^*)  (p_l - N^\varepsilon(\tilde{p}_g)) \, dx\\
   &  -   \frac{1}{\delta t}\int_{\Omega} \Phi \left(\left( {u}  {S} +   {\rho}_g^{\varepsilon}(1- {S})\right) - \left(   u^* S^{*} +   \rho_g^{\varepsilon,*}(1-S^{*})\right)\right) \,  M^\varepsilon(\tilde{p}_g)\, dx \\
   & +\int_{\Omega}  \tilde{\rho}_g^\varepsilon  \lambda_g({S}) \Kb\gb  \cdot\nabla \tilde{p}_g\, dx
     -\int_{\Omega} ({u} {S} +\rho_g^\varepsilon (1-{S})) F_P M^\varepsilon(\tilde{p}_g)\, dx.
 \end{align*}
 By Lemma~\ref{lemma:diss} we have for sufficiently small $\varepsilon$: 
 \begin{align*}
    \Phi {S} \frac{1}{\tilde{\rho}_g^\varepsilon} D\nabla \tilde{{u}} \cdot  \nabla \tilde{p}_g 
    + \lambda_g({S}) \Kb \nabla \tilde{{p}}_g \cdot \nabla \tilde{p}_g
    \geq c_D |\nabla \tilde{u}|^2\\
| \Phi {S} \frac{1}{\tilde{{\rho}}_l}D\nabla \tilde{{u}}\cdot \nabla p_l  | \leq
   \frac{1}{2}\lambda_l^{\varepsilon}({S})\Kb\nabla p_l\cdot \nabla p_l + q c_D |\nabla \tilde{u}|^2,
 \end{align*}
with $0<q<1$, which leads to:
\begin{align*}
  &\int_{\Omega} [  (\frac{1}{\sigma}- \frac{1}{2}) \lambda_l^{\varepsilon}({S})\Kb\nabla p_l\cdot \nabla p_l
  +   (1-q) c_D |\nabla \tilde{u}|^2  + \frac{\varepsilon}{\sigma} \nabla p_g\cdot \nabla \tilde{p}_g ]\, dx\\ 
       &+\eta\int_{\Omega}[ \nabla \tilde{p}_g -  \nabla\tilde{p}_l]  
       \cdot ( \nabla \tilde{p}_g - \nabla p_l )\,  dx
       \leq RHS,
 \end{align*}
where
\begin{align*}
   RHS =& - \int_{\Omega} {S} F_P (p_l - N^\varepsilon(\tilde{p}_g))\, dx + \int_{\Omega} F_I  (p_l - N^\varepsilon(\tilde{p}_g)) \, dx \\
   &-\int_{\Omega} ({u} {S} +\rho_g^\varepsilon (1-{S})) F_P M^\varepsilon(\tilde{p}_g)\, dx 
        -   \frac{1}{\delta t}\int_{\Omega}\Phi  ({S}-S^*)  (p_l - N^\varepsilon(\tilde{p}_g)) \, dx\\
   &  -   \frac{1}{\delta t}\int_{\Omega} \Phi \left(\left( {u}{S} +  \rho_g^{\varepsilon}(1- {S})\right) - \left(   u^* S^{*} + \rho_g^{\varepsilon,*} (1-S^{*})\right)\right) \,  M^\varepsilon(\tilde{p}_g)\, dx \\
   & +\int_{\Omega}{\rho}_l\lambda_l({S})\Kb \gb \cdot \nabla p_l\,  dx+\int_{\Omega}  \tilde{\rho}_g^\varepsilon  \lambda_g({S})  \Kb\gb  \cdot \nabla \tilde{p}_g\, dx.
    \end{align*}
    Using orthogonality of the spectral functions \Pb{2}, multiplying by $\sigma$ and using $\sigma\leq 1$,   we get 
\begin{align*}
  \int_{\Omega} [  \frac{1}{2}  \varepsilon k_{m} |\nabla p_l|^2
  +   \sigma (1-q) c_D |\nabla \tilde{u}|^2  + \varepsilon |\nabla \tilde{p}_g|^2 ]\, dx
       +\eta\sigma \int_{\Omega} |\nabla \tilde{p}_g -  \nabla\tilde{p}_l|^2 \,  dx
       \leq |RHS|. 
 \end{align*}
Since we need an estimate independent of $\sigma$ it is enough to consider 
\begin{align}
\frac{1}{2}  \varepsilon \int_{\Omega} [    k_{m} |\nabla p_l|^2 +  |\nabla \tilde{p}_g|^2 ]\, dx \leq |RHS|.\label{sigma-est-3}
 \end{align}

 In the estimates of the RHS we use  boundedness of the coefficients and  bounds for function $M^{\varepsilon}$ and $N^{\varepsilon}$ 
 given in (\ref{M-N-bounds-neg}).  
For example, we can estimate:
\begin{align*}
  | \int_{\Omega} {S} F_P (p_l - N^\varepsilon(\tilde{p}_g))\, dx | &\leq \int_{\Omega} \left(F_P |p_l|
  + F_P \frac{u_{max}}{\varepsilon} |\tilde{p}_g| \right)\, dx \\
  &\leq \tilde{\varepsilon} \|\nabla p_l\|_{L^2(\Omega)}^2 +\tilde{\varepsilon} \|\nabla \tilde{p}_g\|_{L^2(\Omega)}^2  + C  \|F_P\|_{L^2(\Omega)}^2
\end{align*}
with $\tilde{\varepsilon}$ small enough, depending on $\varepsilon$,
and $C=C(u_{max},C_\Omega,\varepsilon)$, where $C_\Omega$ is the constant from the Poincar\'e inequality.
Note that $C$ is independent of $N$ and $\eta$. 
All the other integrals can be treated in similar way, obtaining 
\begin{align*}
  | RHS | \leq  \tilde{\varepsilon} \|\nabla p_l\|_{L^2(\Omega)}^2 +\tilde{\varepsilon} \|\nabla \tilde{p}_g\|_{L^2(\Omega)}^2  +C,
\end{align*}
where the constant $C$ depends on $\varepsilon$, but it is independent of $\sigma$, $N$ and $\eta$.
As a consequence we get from (\ref{sigma-est-3}) (for $ \tilde{\varepsilon}$ sufficiently small)
\begin{align}
\frac{1}{2}  \varepsilon    \int_{\Omega} [    k_{m} |\nabla p_l|^2  + |\nabla \tilde{p}_g|^2 ]\, dx \leq C,
      \label{sigma-est-4}
 \end{align}
with $C$  independent of $\sigma$, $N$ and $\eta$.

By setting $\psi = p_g$ in  (\ref{eqsystem2}) we get
\begin{align*}
    &\int_{\Omega}
     \varepsilon \tilde{{\rho}}_g^{\varepsilon} |\nabla p_g|^2 \, dx \\
     & =
     -\int_{\Omega}\left(
    \tilde{{u}} \lambda_l^{\varepsilon}({S})\Kb  \nabla p_l
    + \sigma \tilde{{\rho}}_g^{\varepsilon} \lambda_g({S}) \Kb \nabla \tilde{{p}}_g 
    +\sigma \Phi {S}\frac{\rho_l^{std}}{\tilde{{\rho}}_l } D\nabla  \tilde{{u}} 
    \right)
    \cdot  \nabla p_g \, dx \\
      & -\frac{\sigma}{\delta t}\int_{\Omega} \Phi \left(\left( {u}  {S} +   {\rho}_g^{\varepsilon}(1- {S})\right) 
      - \left(   u^* S^{*} +   \rho_g^{\varepsilon,*}(1-S^{*})\right)\right) \, p_g\, dx \\
     &- \sigma\eta\int_{\Omega}\left(\tilde{{\rho}}_g^{\varepsilon} -  \tilde{{u}}\right)
     \left( \nabla \tilde{p}_g -\nabla \tilde{p}_l\right)  \cdot  \nabla p_g\, dx
      - \sigma \int_{\Omega} ({u} {S} +\rho_g^\varepsilon (1-{S})) F_P p_g\, dx \\ 
      & + \sigma \int_{\Omega}\left(  {\rho}_l \tilde{u} \lambda_l({S})+  (\tilde{\rho}_g^\varepsilon)^2  \lambda_g({S}) \right) \Kb\gb  \cdot   \nabla p_g\, dx.
 \end{align*}
Using H\"older and Poincar\'e inequalities we get for any $\tilde{\varepsilon} >0$,
\begin{align*}
  &\varepsilon^2      \int_{\Omega} |\nabla p_g|^2 \, dx 
  \leq  \tilde{\varepsilon} \int_{\Omega} |\nabla p_g|^2 \, dx + C\left( \int_{\Omega}  \left( |\nabla p_l|^2 +  |\nabla \tilde{{u}}|^2
  +  |\nabla \tilde{{p}}_g|^2+  |\nabla \tilde{{p}}_l|^2  \right) \, dx  + 1\right)
  \end{align*}
Using $\lVert \nabla \tilde{u} \lVert _{L^2(\Omega)} \leq M_g \lVert \nabla \tilde{p}_g \lVert _{L^2(\Omega)} $,
$\lVert \nabla \tilde{p}_l \lVert _{L^2(\Omega)} \leq \lVert \nabla {p}_l \lVert _{L^2(\Omega)} $  and (\ref{sigma-est-4}) we obtain
\begin{align}
  \int_{\Omega} |\nabla p_g|^2 \, dx   \leq  C, \label{sigma-est-5}
\end{align}
where $C$ depends on $\varepsilon$ but it is independent of $\sigma$, $N$ and $\eta$.
From (\ref{sigma-est-4}) and (\ref{sigma-est-5}) we conclude that all assumptions of the Schauder fixed point theorem 
are satisfied which proves   the following proposition:
\begin{proposition}
	\label{Prop1}
For given $(p_l^*, p_g^*)\in L^2(\Omega)\times L^{2}(\Omega)$ there exists 
$(p_l, p_g)\in V \times V$ that solve (\ref{eqsystem1-nl}),  (\ref{eqsystem2-nl}):

\begin{align}\nonumber
  \frac{1}{\delta t}\int_{\Omega}\Phi &  ({S}-S^*)\varphi \, dx 
  +\int_{\Omega} [ \lambda_l^{\varepsilon}({S})\Kb\nabla p_l
    - \Phi {S} \frac{1}{\rho_l^{std} + \hat{u}({\EuScript P}_N[ p_g]) }D\nabla \hat{u}({\EuScript P}_N[ p_g]) ]\cdot \nabla \varphi \, dx\\ \label{eqsystem1-nl}
       &-\eta\int_{\Omega}[ \nabla {\EuScript P}_N[ {p}_g] -  \nabla {\EuScript P}_N[ {p}_l]]  \cdot \nabla \varphi\,  dx
       + \int_{\Omega} {S} F_P \varphi\, dx \\ \nonumber 
       &= \int_{\Omega} F_I \varphi \, dx  +
       \int_{\Omega}{\rho}_l\lambda_l({S})\Kb \gb \cdot \nabla \varphi\,  dx
 \end{align}
for all $\varphi\in V$ and
\begin{align}\nonumber
  \frac{1}{\delta t}\int_{\Omega} \Phi (\big( {u}  {S} &+   {\rho}_g^\varepsilon (1- {S})) 
  - (   u^* S^{*} + \rho_g^{\varepsilon, *}(1-S^{*}))\big) \, \psi\, dx 
      +\int_{\Omega}   \hat{u}({\EuScript P}_N[ p_g]) \lambda_l^{\varepsilon}({S})\Kb  \nabla p_l     \cdot \nabla \psi \, dx                                        \\
    & +\int_{\Omega}\left(
    \hat{\rho}_g^\varepsilon ({\EuScript P}_N[ p_g]) \lambda_g({S}) \Kb \nabla {\EuScript P}_N[ p_g] 
    + \varepsilon \hat{\rho}_g^\varepsilon ({\EuScript P}_N[ p_g])  \nabla p_g
    \right)
    \cdot \nabla \psi \, dx \nonumber\\
     & +\int_{\Omega}
    \Phi {S}\frac{\rho_l^{std}}{\rho_l^{std} + \hat{u}({\EuScript P}_N[ p_g]) } D\nabla  \hat{u}({\EuScript P}_N[ p_g])
    \cdot \nabla \psi \, dx \label{eqsystem2-nl}\\
     \nonumber
     & + \eta\int_{\Omega}\left(\hat{\rho}_g^\varepsilon ({\EuScript P}_N[ p_g])  -  \hat{u}({\EuScript P}_N[ p_g])\right)
     \left( \nabla{\EuScript P}_N[ p_g]-\nabla {\EuScript P}_N[ p_l]\right)  \cdot \nabla \psi\, dx\\
    &  + \int_{\Omega} ({u} {S} +\rho_g^\varepsilon  (1-{S})) F_P\psi\, dx \nonumber\\ \nonumber
   & =\int_{\Omega}\left(  {\rho}_l   \hat{u}({\EuScript P}_N[ p_g]) \lambda_l({S})+  (\hat{\rho}_g^\varepsilon ({\EuScript P}_N[ p_g]))^2  \lambda_g({S}) \right) 
   \Kb\gb  \cdot \nabla \psi dx
 \end{align}
for all $\psi \in V$. The secondary variables in equations (\ref{eqsystem1-nl}),  (\ref{eqsystem2-nl}) are given by:
\[ 
   S=p_c^{-1}(p_g-p_l), \quad u = \hat{u}(p_g),\quad \rho_g^\varepsilon = \hat{\rho}_g^\varepsilon(p_g),
   \quad \rho_l = \rho_l^{std} + \hat{u}(p_g).
   \]
\end{proposition} 
Note that $p_l$ and $p_g$ depend on $\eta$, $\varepsilon$ and $N$. However, we omit this dependency in writing for simplicity 
until passing to the limit in some of the parameters, when parameter of interest will be denoted explicitly.

\subsection{Step 2. Limit as $N\to\infty$}

By applying a priori estimates (\ref{sigma-est-4}) and (\ref{sigma-est-5}) 
given in the proof of  Proposition~\ref{Prop1} for $\sigma = 1$ we get the following result:
\begin{corollary}\label{cor-1}
  There is a constant $C>0$ independent of $N$ and $\eta$ (but depending  on $\varepsilon$) 
  such that any solutions $(p_l^N, p_g^N)\in V\times V$ to problem  (\ref{eqsystem1-nl}),  (\ref{eqsystem2-nl})  satisfy 
  \[  \int_{\Omega} |\nabla {\EuScript P}_N[p_g^N] |^2 \, dx,\;\int_{\Omega} |\nabla p_g^N|^2 \, dx,
   \;\int_{\Omega} |\nabla p_l^N|^2 \, dx   \leq  C.\]
\end{corollary}

We consider behaviour of   the solution to (\ref{eqsystem1-nl})--(\ref{eqsystem2-nl}), $p_l^N$ and $p_g^N$,  
as $N\to\infty$, while all other  
regularization parameters, $\varepsilon$, $\eta$ and $\delta t$, are kept constant. 
We also denote the secondary variables as
\[  u^N = \hat{u}(p_g^N),\quad \rho_g^N = \hat{\rho}_g^\varepsilon(p_g^N),\quad \rho_l^N = \rho_l^{std} + \hat{u}(p_g^N), \quad S^N = p_c^{-1}(p_g^N - p_l^N).\]
The uniform bounds (with respect to $N$) from Corollary~\ref{cor-1} imply that there is a subsequence, still denoted by $N$, such that
 as $N\to\infty$,
\begin{align*}
  p_g^N &\to p_g \quad \text{weakly in $V$, strongly in}\; L^2(\Omega)\; \text{ and a.e. in} \; \Omega,\\
  p_l^N &\to p_l \quad \text{weakly in $V$, strongly in}\; L^2(\Omega)\; \text{ and a.e. in} \; \Omega,\\
  {\EuScript P}_N[p_g^N] &\to \xi \quad \text{ weakly in $V$, strongly in}\; L^2(\Omega)\; \text{ and a.e. in} \; \Omega,
\end{align*}
for some $p_l, p_g, \xi \in V$. Using property \Pb{3} of the projection operator 
we find that $\xi = p_g$.  Due to properties \Hb{5} and \Hb{6} we have 
\begin{align*}
  S^N \to S = p_c^{-1}(p_g - p_l)\quad \text{ a.e. in } \; \Omega,\\
  u^N \to u = \hat{u}(p_g) \quad \text{weakly in $V$  and a.e. in } \; \Omega,\\
  \rho_g^N \to \rho_g^\varepsilon = \hat{\rho}_g^\varepsilon(p_g) \quad \text{ a.e. in } \; \Omega,\\
  1/\rho_l^N \to 1/\rho_l = 1/(\rho_l^{std} + \hat{u}(p_g)) \quad \text{ a.e. in } \; \Omega.
\end{align*}
This convergences are sufficient to pass to the limit as $N\to \infty$ in (\ref{eqsystem1-nl})--(\ref{eqsystem2-nl}), and we get
\begin{align}\nonumber
  \frac{1}{\delta t}\int_{\Omega}\Phi &  ({S}-S^*)\varphi \, dx 
  +\int_{\Omega} [ \lambda_l^{\varepsilon}({S})\Kb\nabla p_l
    - \Phi {S} \frac{1}{{{\rho}}_l}D\nabla \hat{u}(p_g) ]\cdot \nabla \varphi \, dx\\ \label{eqsystem1-nl-N0}
       &-\eta\int_{\Omega}[ \nabla {p}_g -  \nabla {p}_l ]  \cdot \nabla \varphi\,  dx
       + \int_{\Omega} {S} F_P \varphi\, dx \\ \nonumber 
       &= \int_{\Omega} F_I \varphi \, dx  +
       \int_{\Omega}{\rho}_l\lambda_l({S})\Kb \gb \cdot \nabla \varphi\,  dx
 \end{align}
for all $\varphi\in V$ and
\begin{align}\nonumber
  \frac{1}{\delta t}&\int_{\Omega} \Phi \left(\left( {u}  {S} +   {\rho}_g^\varepsilon(1- {S})\right)
                                             - \left(   u^* S^{*} +   \rho_g^{\varepsilon,*}(1-S^{*})\right)\right) \, \psi\, dx \\
    & +\int_{\Omega}\left( {u} \lambda_l^{\varepsilon}({S})\Kb  \nabla p_l
    +  {\rho}_g^\varepsilon \lambda_g({S}) \Kb \nabla {{p}}_g +  \varepsilon {\rho}_g^\varepsilon \nabla {{p}}_g 
    +  \Phi {S}\frac{\rho_l^{std}}{\rho_l} D\nabla  {u} \right)
    \cdot \nabla \psi \, dx \label{eqsystem2-nl-N0}\\
     \nonumber
     & + \eta\int_{\Omega}\left({\rho}_g^\varepsilon  -  {u}\right) \left( \nabla p_g-\nabla p_l\right)  \cdot \nabla \psi\, dx
      + \int_{\Omega} ({u} {S} +\rho_g^\varepsilon (1-{S})) F_P\psi\, dx \nonumber\\ \nonumber
   & =\int_{\Omega}\left(  {\rho}_l{u} \lambda_l({S})+  ({\rho}_g^\varepsilon)^2  \lambda_g({S}) \right) \Kb\gb  \cdot \nabla \psi dx
 \end{align}
for all $\psi \in V$, where
\begin{align}\label{eqsystem3-nl-N0}
    u = \hat{u}(p_g),\quad \rho_g^\varepsilon = \hat{\rho}_g^\varepsilon(p_g),\quad \rho_l = \rho_l^{std} + \hat{u}(p_g), \quad S = p_c^{-1}(p_g - p_l).
\end{align}

We have proved the following result. 
\begin{proposition}
For given $(p_l^*, p_g^*)\in L^2(\Omega)\times L^{2}(\Omega)$ there exists 
$(p_l, p_g)\in V \times V$ that solve problem (\ref{eqsystem1-nl-N0}), (\ref{eqsystem2-nl-N0}) and (\ref{eqsystem3-nl-N0}). 
  \label{lema-egz-eps}
\end{proposition}

\subsection{Step 3. Limit as $\varepsilon\to 0$ }


For passage to the limit as  $\varepsilon\to 0$ we need to refine a priori estimates 
since they are not independent of  $\varepsilon$. This will be achieved by using the test functions
$\varphi = p_l - N^\varepsilon(p_g)$ and $\psi = M^\varepsilon(p_g)$ in (\ref{eqsystem1-nl-N0}) and (\ref{eqsystem2-nl-N0}),
which lead to the following estimate:

\begin{lemma}
 There is a constant $C$ independent of $\delta t$, $\eta$ and $\varepsilon$ such that each solution to the problem 
(\ref{eqsystem1-nl-N0}), (\ref{eqsystem2-nl-N0}) and (\ref{eqsystem3-nl-N0}) satisfies: 
\begin{align}
 & \frac{1}{\delta t}\int_{\Omega}\Phi [  {\cal E}^\varepsilon(p_l, p_g) -  {\cal E}^\varepsilon(p_l^*, p_g^*) ]\, dx \nonumber\\
 &  + \int_{\Omega} [ \lambda_l({S}) \Kb\nabla p_l \cdot \nabla p_l 
  +\lambda_g(S) \Kb \nabla p_g \cdot \nabla p_g + c_D |\nabla u|^2  + \varepsilon |\nabla p_g|^2 ]\, dx \label{uni-est}\\
& + \eta \int_{\Omega} |\nabla p_g - \nabla p_l|^2 \, dx
 \leq  C ,\nonumber
\end{align}
where the function ${\cal E}^\varepsilon$ is given by 
\begin{equation}
 \label{E-final-delta}
\begin{split}
  {\cal E}^\varepsilon(p_l, p_g) &=   S[ \hat{u}(p_g)  M^\varepsilon(p_g) - N^\varepsilon(p_g)] 
  + (1-S)  [\hat{\rho}^\varepsilon_g(p_g) M^\varepsilon(p_g) - p_g ] \\
  &-\int_0^S  p_c(\sigma)\,d\sigma.
\end{split}
\end{equation}
  \label{lema-uni-est}
\end{lemma}
\begin{proof}

After introducing test functions $\varphi = p_l - N^\varepsilon(p_g)$ and $\psi = M^\varepsilon(p_g)$ in (\ref{eqsystem1-nl-N0}) and (\ref{eqsystem2-nl-N0})
and summation of the two equations, we get the following equation: 
\begin{align}\nonumber
\frac{1}{\delta t}\int_{\Omega}\Phi  ({S} & -S^*) (p_l - N^\varepsilon(p_g))  \, dx \\ 
\nonumber
& +\frac{1}{\delta t}\int_{\Omega}  \Phi \left[\left( {u}  {S} +   \hat{\rho}^\varepsilon_g(p_g)(1- {S})\right)
	- \left(   u^* S^{*} +   \rho_g^{\varepsilon,*}(1-S^{*})\right)\right] \, M^\varepsilon(p_g)\, dx \\ 
	\label{eqsystem2-sum}
& +\int_{\Omega} [ \lambda_l^{\varepsilon}({S}) \Kb\nabla p_l \cdot \nabla p_l - \Phi \frac{S}{{{\rho}}_l}D\nabla u \cdot\nabla p_l 
  + \lambda_g({S}) \Kb \nabla {{p}}_g \cdot \nabla p_g + \varepsilon |\nabla p_g|^2   ] \, dx\\ 
	  \nonumber
& +\int_{\Omega} \frac{\Phi S D}{\rho_g^\varepsilon}\nabla u \cdot \nabla p_g \, dx
 + \eta \int_{\Omega} |\nabla p_g - \nabla p_l|^2 \, dx  
	 = RHS ,
\end{align}
where $RHS = I_1 + I_2 + I_3$ with
\begin{align}
I_1 & = \int_{\Omega} F_I (p_l - N^\varepsilon(p_g))\, dx \nonumber \\
I_2 &= -  \int_{\Omega} {S} F_P (p_l - N^\varepsilon(p_g))\, dx 
       - \int_{\Omega} ({u} {S} +\rho_g^\varepsilon (1-{S})) F_P M^\varepsilon(p_g)\, dx \label{eqsystem3-sum}\\
I_3  & =  \int_{\Omega}{\rho}_l\lambda_l({S})\Kb \gb \cdot \nabla p_l \,  dx
	+ \int_{\Omega} \rho_g^\varepsilon  \lambda_g({S}) \Kb\gb  \cdot \nabla p_g dx.\nonumber
\end{align}

First we consider the accumulation terms in (\ref{eqsystem2-sum}) in which we will use shorthand notation:
\begin{align*}
  {\cal J} =   ({S}-S^*) (p_l &- N^\varepsilon(p_g)) + \left[\left( {u}  {S} +   \hat{\rho}^\varepsilon_g(p_g)(1- {S})\right)
	- \left(   u^* S^{*} +   \rho_g^{\varepsilon,*}(1-S^{*})\right)\right] \, M^\varepsilon(p_g)
\end{align*}
Then by simple manipulations we get:
\begin{align*}
 {\cal J}= S p_l &+ S( u  M^\varepsilon(p_g) - N^\varepsilon(p_g)) + (1-S)  \hat{\rho}^\varepsilon_g(p_g) M^\varepsilon(p_g)\\
	&-\left[ S^* p_l^* + S^*( u^*  M^\varepsilon(p_g^*) - N^\varepsilon(p_g^*)) + (1-S^*)  {\rho}^{\varepsilon,*}_g M^\varepsilon(p_g^*) \right] \\
&+ S^*(p_l^* - p_l) + S^* \left(  [ u^* M^\varepsilon(p_g^*)- N^\varepsilon(p_g^*) ] - [u^* M^\varepsilon(p_g)- N^\varepsilon(p_g) ] \right)\\
&+ (1-S^*){\rho}^{\varepsilon,*}_g [ M^\varepsilon(p_g^*) -  M^\varepsilon(p_g) ].
\end{align*}
Note that from \Hb{5} we get
\begin{align*}
  [ u^* M^\varepsilon(p_g^*)- N^\varepsilon(p_g^*) ] - [u^* M^\varepsilon(p_g)- N^\varepsilon(p_g) ] = 
  \hat{u}(p_g^*) \int_{p_g}^{p_g^*} \frac{d\sigma}{\hat{\rho}^\varepsilon_g(\sigma)} -  
  \int_{p_g}^{p_g^*} \frac{\hat{u}(\sigma)}{\hat{\rho}^\varepsilon_g(\sigma)}d\sigma\geq 0,
\end{align*}
and from \Hb{6}
\begin{align*}
   (1-S^*){\rho}^*_g [ M^\varepsilon(p_g^*) -  M^\varepsilon(p_g) ] = (1-S^*) \hat{\rho}^\varepsilon_g(p_g^*) 
   \int_{p_g}^{p_g^*} \frac{d\sigma}{\hat{\rho}^\varepsilon_g(\sigma)} \geq (1-S^*) (p_g^* - p_g),
\end{align*}
leading to 
\begin{align*}
{\cal J}\geq   S( u&  M^\varepsilon(p_g) - N^\varepsilon(p_g)) + (1-S)  (\hat{\rho}^\varepsilon_g(p_g) M^\varepsilon(p_g) - p_g ) \\
&-\left[ S^* ( u^*  M^\varepsilon(p_g^*) - N^\varepsilon(p_g^*)) + (1-S^*)  ({\rho}^{\varepsilon,*}_g M^\varepsilon(p_g^*) -p_g^* ) \right]
+ (S^* -S) (p_g - p_l).
\end{align*}
Using \Hb{4} one can estimate
\begin{align*}
  (S^* -S) (p_g - p_l) \geq  (S^* -S) p_c(S) \geq \int_S^{S^*} p_c(\sigma)\,d\sigma,
\end{align*}
and therefore we can estimate the accumulation term as follows:
\begin{equation}
  \label{E-delta-est}
  \frac{1}{\delta t}\int_{\Omega}\Phi  {\cal J}\, dx  \\
	\geq  \frac{1}{\delta t}\int_{\Omega}\Phi [  {\cal E}^\varepsilon(p_l, p_g) -  {\cal E}^\varepsilon(p_l^*, p_g^*) ]\, dx.
\end{equation}
where the function ${\cal E}^\varepsilon$ is given by (\ref{E-final-delta}).

We consider now the third and the fourth integrals in (\ref{eqsystem2-sum}).  Applying Lemma~\ref{lemma:diss} we get
\begin{align*}
  & \frac{\Phi S D}{\rho_g^\varepsilon} \nabla u \cdot  \nabla p_g + \lambda_g(S) \Kb \nabla p_g \cdot \nabla p_g \geq c_D  |\nabla u|^2 \\
&|\Phi S D \frac{1}{\rho_l}\nabla u \cdot \nabla p_l|\leq \frac{1}{2}\lambda_l^\varepsilon(S) \Kb\nabla p_l\cdot \nabla p_l + q c_D  |\nabla u|^2.
\end{align*}
If we denote the sum of the third and the fourth integral in (\ref{eqsystem2-sum}) by ${\cal I}$, then we easily get:
\begin{equation}
 \label{main-est-1}
\begin{split}
  {\cal I}  &\geq  \int_{\Omega} [\frac{1}{2} \lambda_l^{\varepsilon}({S}) \Kb\nabla p_l \cdot \nabla p_l 
  + \frac{1-q}{2}\lambda_g(S) \Kb \nabla p_g \cdot \nabla p_g ]\, dx \\
 &  +\int_{\Omega} [\frac{1-q}{2} c_D |\nabla u|^2  + \varepsilon |\nabla p_g|^2 ]\, dx. 
\end{split}
\end{equation}
Finally, let us estimate the right hand side  in (\ref{eqsystem2-sum}). 
From $F_I\geq 0$, $p_l \leq p$ and since $N^\varepsilon(p_g)\geq 0$ for $p_g\in \Rb$ we can estimate
\begin{align}
	I_1 = \int_{Q_T} F_I (p_l -N^\varepsilon(p_g))  dx dt \leq \int_{Q_T} F_I p\, dx dt
\leq  C_1 + \frac{\tilde{\varepsilon}}{2} \| p\|_{L^2(Q_T)}^2,
\label{pard-est-rhs-2a}
\end{align}
for an arbitrary $\tilde{\varepsilon}$, and  $C_1= C_1( \| F_I\|_{L^2(Q_T)},\tilde{\varepsilon})$.
 
The term $I_2$ can be rearranged as follows:
\begin{align*}
I_2 =  -  \int_{\Omega} {S} F_P p_l\, dx 
         + \int_{\Omega} {S} F_P  (N^\varepsilon(p_g) - \hat{u}(p_g) M^\varepsilon(p_g))\, dx 
       - \int_{\Omega}  \rho_g^\varepsilon (1-{S}) F_P M^\varepsilon(p_g)\, dx.
\end{align*} 
Since the function $\hat{u}$ is nondecreasing on $\Rb$ we have $N^\varepsilon(p_g) - \hat{u}(p_g)M^\varepsilon(p_g)\leq 0$
 and $F_P \geq 0$ gives
\begin{align}
\int_{Q_T} S F_P \left( N^\varepsilon(p_g) - \hat{u}(p_g)M^\varepsilon(p_g) \right)  dx dt  \leq 0.
\end{align} 
 
From Lemma~\ref{lema-GP-bounds}  we can estimate the terms with the liquid pressure by
the global pressure as follows: 
\begin{align*}
 -\int_{Q_T} S F_P p_l dx dt   \leq \int_{Q_T} F_P (|p|+C) dx dt 
 \leq C_2 + \frac{\tilde{\varepsilon}}{4} \| p\|_{L^2(Q_T)}^2 ,
\end{align*}
for some $\tilde{\varepsilon} >0$ and $C_2= C_2( \| F_P\|_{L^2(Q_T)},\tilde{\varepsilon})$.

The last term in $I_2$ is non positive for $p_g\geq 0$, and in the region where $p_g < 0$ by Lemma~\ref{lema-GP-bounds}  it holds:
\begin{align*}
-\int_{Q_T} \rho_g^\varepsilon(p_g) S_g F_P M^\varepsilon(p_g) dx dt & =  \int_{Q_T} F_P |S_g p_g| dx dt  \\
& \leq \int_{Q_T} F_P (|p| + C) dx dt \leq C_3  + \frac{\tilde{\varepsilon}}{4} \| p\|_{L^2(Q_T)}^2,
\end{align*}
for arbitrary $\tilde{\varepsilon} >0$ and $C_3 = C_3( \| F_P\|_{L^2(Q_T)},\tilde{\varepsilon})$. 
Therefore, we conclude that for arbitrary $\tilde{\varepsilon} >0$ we have the estimate:
\begin{equation}
\label{pard-est-rhs-1c}
I_2 \leq C_4 +  \frac{\tilde{\varepsilon}}{2} \| p\|_{L^2(Q_T)}^2,
\end{equation}
where $C_4 = C_4( \| F_P\|_{L^2(Q_T)},\tilde{\varepsilon})$.

A straightforward estimate, based on boundedness of the gas and the liquid densities  gives:
\begin{align}
\label{pard-est-rhs-4a1}
I_3 \leq C_5 + \hat{\varepsilon}\int_{Q_T}\lambda_g(S)\Kb\nabla p_g\cdot \nabla p_g\, dx dt
      + \hat{\varepsilon} \int_{Q_T}\lambda_l(S)\Kb\nabla p_l\cdot \nabla p_l\, dx dt,
\end{align}
for an arbitrary $\hat{\varepsilon}$. 

The global pressure norm can be estimated by the Poincar\'e inequality and Lemma~\ref{lema-fund-gl1} as follows:
\begin{align}
   \| p\|_{L^2(Q_T)}^2 &\leq C\int_{\Omega}  (\lambda_l({S}) \Kb\nabla p_l \cdot \nabla p_l  +\lambda_g(S) \Kb \nabla p_g \cdot \nabla p_g)\, dx.
   \label{glop-est}
\end{align}

From  estimates (\ref{E-delta-est}),  (\ref{main-est-1}), (\ref{pard-est-rhs-2a}),  (\ref{pard-est-rhs-1c}), (\ref{pard-est-rhs-4a1}) and
(\ref{glop-est}), taking sufficiently small $\tilde{\varepsilon}$ and $\hat{\varepsilon}$  we obtain the estimate (\ref{uni-est}). 
Lemma~\ref{lema-uni-est} is proved.
\end{proof}

\begin{remark}\label{remark-bound}
Note that by using Lemma~\ref{fund-gl1} we can write estimate (\ref{uni-est}) also as follows:
\begin{align*}
 & \frac{1}{\delta t}\int_{\Omega}\Phi [  {\cal E}^\varepsilon(p_l, p_g) -  {\cal E}^\varepsilon(p_l^*, p_g^*) ]\, dx \\
 &  + \int_{\Omega} [ \lambda({S}) \Kb\nabla p \cdot \nabla p 
  +\Kb \nabla \beta(S) \cdot \nabla \beta(S) + c_D |\nabla u|^2  + \varepsilon |\nabla p_g|^2 ]\, dx \\
& + \eta \int_{\Omega} |\nabla p_g - \nabla p_l|^2 \, dx
 \leq  C .
\end{align*}
\end{remark}

Due to the monotonicity of function $\hat{u}$ and definition of function $\hat{\rho}_g^{\varepsilon} $  
we can carry out the same steps as in the proof of Lemma~\ref{lema-E-est} to show
\begin{align}
\label{E-est-lower}
{\cal E}^\varepsilon(p_l, p_g) \geq -M_{p_c},
\end{align}
for $p_l, p_g  \in \mathbb{R}$. Also, we have the  upper bound 
\begin{align}
\label{E-est-upper}
{\cal E}^\varepsilon(p_l^*, p_g^*) \leq C( p_g^* + 1)
\end{align}
since $p_g^*$ satisfies $p_g^* \geq 0$. 
We can apply previous estimates (\ref{E-est-lower}) and (\ref{E-est-upper}) to the estimate (\ref{uni-est})
and obtain that each solution to the problem (\ref{eqsystem1-nl-N0}), (\ref{eqsystem2-nl-N0}) and (\ref{eqsystem3-nl-N0})
with $p_g^* \geq 0$  satisfy the following bound:

\begin{equation}
 \label{uni-est-gl}
\begin{split}
	 &  \int_{\Omega} [ |\nabla p|^2 + |\nabla \beta(S)|^2 + |\nabla u|^2  + \varepsilon |\nabla p_g|^2 ]\, dx
	 + \eta \int_{\Omega} |\nabla p_g - \nabla p_l|^2 \, dx 	 \leq  C ,
\end{split}
\end{equation}
where the constant $C$ is independent of $\varepsilon$ and $\eta$. 

We shall now denote the solution  to the problem 
(\ref{eqsystem1-nl-N0}), (\ref{eqsystem2-nl-N0}) and (\ref{eqsystem3-nl-N0})
by $p_l^\varepsilon$ and $p_g^\varepsilon$. All secondary variables will also be denoted by $\varepsilon$:
\begin{align}\label{secondary-eps}
    u^\varepsilon = \hat{u}(p_g^\varepsilon),\quad 
    \rho_g^\varepsilon = \hat{\rho}_g^\varepsilon(p_g^\varepsilon),\quad \rho_l^\varepsilon = \rho_l^{std} + \hat{u}(p_g^\varepsilon), 
    \quad S^\varepsilon = p_c^{-1}(p_g^\varepsilon - p_l^\varepsilon),
\end{align}
and the global pressure defined by (\ref{glrelation}) is denoted $p^\varepsilon$.

The bounds (\ref{uni-est}) and  (\ref{uni-est-gl}) give the following  bounds uniform with respect to $\varepsilon$:
\begin{align}
(u^{\varepsilon})_\varepsilon \text{ is uniformly bounded in } V, \label{uni_est_eps1}\\
(p^{\varepsilon})_\varepsilon \text{ is uniformly bounded in } V, \label{uni_est_eps2}\\
(\beta(S^{\varepsilon}))_\varepsilon \text{ is uniformly bounded in } H^1(\Omega), \label{uni_est_eps3}\\
(\sqrt{\varepsilon}\nabla p_l^{\varepsilon})_\varepsilon \text{ is uniformly bounded in } L^2(\Omega), \label{uni_est_eps5}\\
(\sqrt{\varepsilon}\nabla p_g^{\varepsilon})_\varepsilon \text{ is uniformly bounded in } L^2(\Omega), \label{uni_est_eps6}\\
(\nabla p_c(S^{\varepsilon}))_\varepsilon \text{ is uniformly bounded in } L^2(\Omega). \label{uni_est_eps7}
\end{align}

\begin{lemma}
\label{conv_eps}
Let $p_l^\varepsilon$ and $p_g^\varepsilon$ be the solution to
 (\ref{eqsystem1-nl-N0}), (\ref{eqsystem2-nl-N0}) and (\ref{eqsystem3-nl-N0})
and let corresponding secondary variables be denoted as in  (\ref{secondary-eps}). Then
there exist functions $p_l, p_g \in L^2(\Omega)$,  $S = p_c^{-1}(p_g - p_l)$ and $p = p_l + \overline{P}(S)\in V$
such that on a subsequence it holds:
\begin{align}
p_l^{\varepsilon}\longrightarrow p_l\quad \text{ a.e in } \Omega, \label{conv_eps7}\\
p_g^{\varepsilon} \longrightarrow p_g\quad \text{ a.e in } \Omega, \label{conv_eps7a}\\
S^{\varepsilon}\longrightarrow S \quad\text{ a.e. in } \Omega, \label{conv_eps5}\\
\rho_g^{\varepsilon} = \hat{\rho}_g^{\varepsilon} (p_g^{\varepsilon}) \longrightarrow \rho_g = \hat{\rho}_g(p_g) \quad\text{ a.e. in } \Omega, \label{conv_eps9}\\
\rho_l^{\varepsilon} = \rho_l^{std} + \hat{u}(p_g^{\varepsilon}) \longrightarrow \rho_l = \rho_l^{std} + \hat{u}(p_g) \quad\text{ a.e. in } \Omega, \label{conv_eps10}\\
 u^\varepsilon \longrightarrow u= \hat{u}(p_g)\quad \text{ weakly in } V \text{ and a.e. in } \Omega, \label{conv_eps1}\\
p^\varepsilon \longrightarrow p\quad \text{ weakly in } V \text{ and a.e. in } \Omega ,\label{conv_eps2}\\
\beta(S^{\varepsilon})\longrightarrow  \beta(S)\quad \text{ weakly in } H^1(\Omega) \text{ and a.e. in } \Omega, \label{conv_eps3}\\
 p_c(S^{\varepsilon}) \longrightarrow p_c(S)\quad \text{ weakly in } H^1(\Omega). \label{conv_eps4}
\end{align}
\end{lemma}
\begin{proof}
  Convergence (\ref{conv_eps2}) follows directly from (\ref{uni_est_eps2}). 
From (\ref{uni_est_eps7}) and the Dirichlet boundary condition we conclude that  $\hat p_c(S^{\varepsilon})\longrightarrow \xi$ weakly in $ H^1(\Omega)$
and a.e. in $\Omega$, for some $\xi\in H^1(\Omega)$, $\xi\geq 0$.
Since the function $p_c$ is invertible we can define $S=p_c^{-1}(\xi)$ and now (\ref{conv_eps4}) and (\ref{conv_eps5}) follow.
From  (\ref{uni_est_eps3}) and  (\ref{conv_eps5}) we obtain (\ref{conv_eps3}).

Definition of the global pressure gives 
\begin{align*}
  p_l^\varepsilon = p^\varepsilon +\int_{S^\varepsilon}^1 \frac{\lambda_g(s)}{\lambda_l(s)+\lambda_g(s)}p_c'(s) ds\longrightarrow
  p +\int_{S}^1 \frac{\lambda_g(s)}{\lambda_l(s)+\lambda_g(s)}p_c'(s) ds =: p_l, \; \text{a.e. in } \Omega,
\end{align*}
where we define limiting liquid pressure $p_l$ by its relation to the limiting global pressure. Similarly,
\begin{align*}
    p_g^\varepsilon = p_l^\varepsilon +p_c(S^\varepsilon) \longrightarrow p_l +p_c(S) =: p_g \; \text{a.e. in } \Omega.
\end{align*}
Obviously, we have $S = p_c^{-1}(p_g - p_l)$. This proves (\ref{conv_eps7}), (\ref{conv_eps7a}) and (\ref{conv_eps9}) and (\ref{conv_eps10})
follow from the continuity of the functions $\hat{\rho}_g$ and  $\hat u$, and the uniform convergence
of $\hat{\rho}_g^\varepsilon$  towards $\hat{\rho}_g$. Finally, (\ref{conv_eps1}) is a consequence of (\ref{uni_est_eps1}). 
\end{proof}

\subsection{End of the proof of Theorem~\ref{TM_3}}

In order to prove Theorem~\ref{TM_3} we need to pass to the limit as $\varepsilon \to 0$ 
in the equations (\ref{eqsystem1-nl-N0})-(\ref{eqsystem2-nl-N0}) using convergences established in 
Lemma~\ref{conv_eps}. This passage to the limit is evident in all terms except the terms with gradients of the phase 
pressures. In these terms we use relation (\ref{fund-gl1a}).
For example:

\begin{align*}\nonumber
    \int_{\Omega} {u}^\varepsilon \lambda_l^{\varepsilon}({S}^\varepsilon) & \Kb  \nabla p_l^\varepsilon    \cdot \nabla \psi \, dx =
 \int_{\Omega} {u}^\varepsilon [\lambda_l(S^\varepsilon){\Kb}\nabla p^\varepsilon + \gamma(S^\varepsilon){\Kb}\nabla \beta(S^\varepsilon)  ] 
      \cdot \nabla \psi \, dx \\
      &\to \int_{\Omega} {u} [\lambda_l(S){\Kb}\nabla p + \gamma(S){\Kb}\nabla \beta(S)  ]     \cdot \nabla \psi \, dx  
      = \int_{\Omega} {u} \lambda_l(S){\Kb}\nabla p_l     \cdot \nabla \psi \, dx ,
 \end{align*}
 where limit liquid pressure $p_l$ is defined from the limit global pressure $p$ and the limit saturation $S$ by (\ref{glrelation}). 
 In this way we have proved that for given $p_l^{k-1}, p_g^{k-1} \in  V$, $ p_g^{k-1}\geq 0$, there exists at least one 
 solution  $p_l^k, p_g^k \in  V$ of (\ref{tmp_p1-1}) and (\ref{tmp_p1-2}).
In order to finish the  proof of Theorem~\ref{TM_3} we need to prove non-negativity of pseudo gas pressure $p_g^k$. 

\begin{lemma}\label{lemma-pos}
	Let  $p_l^{k-1}, p_g^{k-1} \in  V$,  $p_g^{k-1} \geq 0$. Then solution to the problem (\ref{tmp_p1-1}),  (\ref{tmp_p1-2}) satisfy
	$p_g^k \geq 0.$ 
\end{lemma}
\begin{proof}
Let us define $X=\min(u^k, 0)$. We set $\varphi = X^2/2$ in the liquid phase equation  (\ref{tmp_p1-1}) and $\psi = X$ in 
the gas phase equation (\ref{tmp_p1-2}). Note that integration in these equations   is performed only on the part of the domain where 
$p_g^{k}\leq 0$ which cancels the terms multiplied by  $\rho_g^k$, since  $\hat{\rho}_g(p_g)=0$ for $p_g\leq 0$.
By subtracting the  liquid phase equation  from the gas phase equation we get,
\begin{align*}
	\frac{1}{\delta t}\int_{\Omega} &\Phi \left( X^2 S^{k} - \left(   u^{k-1} S^{k-1} +   \rho_g^{k-1}(1-S^{k-1})\right) X - (S^k - S^{k-1}) 	\frac{X^2}{2}\right)  \, dx  \\
	& + \int_{\Omega}\Phi S^k D\lvert \nabla X\lvert ^2 \, dx + \int_{\Omega} S^k F_P^k \frac{X^2}{2} \, dx = -\int_{\Omega} F_I\frac{X^2}{2} \, dx.
\end{align*}
Due to the fact $p_g^{k-1}\geq 0$ and $X\leq 0$ we have 
\begin{align*}
	- \left(   u^{k-1} S^{k-1} +   \rho_g^{k-1}(1-S^{k-1})\right) X \geq 0
\end{align*}
 which leads to
 
 \begin{align*}
 \frac{1}{\delta t}\int_{\Omega} &\Phi\frac{X^2}{2} \left( S^{k} + S^{k-1}\right)  \, dx + \int_{\Omega}\Phi S^k D\lvert \nabla X\lvert ^2 \, dx + \int_{\Omega} S^k F_P^k \frac{X^2}{2} \, dx \leq -\int_{\Omega} F_I\frac{X^2}{2} \, dx \leq 0
 \end{align*}
 From here we conclude that $X=0$ and Lemma~\ref{lemma-pos} is proved.
 \end{proof}
This completes the proof of Theorem~\ref{TM_3}.

\section{Uniform estimates with respect to $\delta t$}
\label{sec:un-est}

From Lemma~\ref{lema-uni-est} and Remark~\ref{remark-bound} it follows that there exists  constant $C$,  
independent of $\delta t$, $\eta$ and $\varepsilon$  such that each solution to the problem 
(\ref{eqsystem1-nl-N0}), (\ref{eqsystem2-nl-N0}) and (\ref{eqsystem3-nl-N0}) satisfies: 
\begin{align*}
  \frac{1}{\delta t}\int_{\Omega}\Phi [  {\cal E}^\varepsilon(p_l^\varepsilon, p_g^\varepsilon) -  {\cal E}^\varepsilon(p_l^*, p_g^*) ]\, dx 
   + \int_{\Omega} [ |\nabla p^\varepsilon|^2
 & +    |\nabla \beta(S^\varepsilon)|^2 + |\nabla u^\varepsilon|^2  ]\, dx \\ 
& + \eta \int_{\Omega} |\nabla p_g^\varepsilon - \nabla p_l^\varepsilon|^2 \, dx
 \leq  C.\nonumber
\end{align*}
In this inequality $p_g^\varepsilon$ is not necessarily positive, but due to monotonicity of the function $\hat{u}$ we have
${\cal E}^\varepsilon(p_l^\varepsilon, p_g^\varepsilon) \geq {\cal E}^\varepsilon(p_l^\varepsilon, (p_g^\varepsilon)^+)$. 
Then, it is easy to see that 
\begin{align*}
  \int_{\Omega}\Phi   {\cal E}^\varepsilon(p_l^\varepsilon, (p_g^\varepsilon)^+)\, dx \longrightarrow   \int_{\Omega}\Phi   {\cal E}(p_l, p_g)\, dx\\
 \int_{\Omega}\Phi {\cal E}^\varepsilon(p_l^*, p_g^*) \, dx \longrightarrow   \int_{\Omega}\Phi   {\cal E}(p_l^*, p_g^*)\, dx
\end{align*}
as $\varepsilon\to 0$, where $p_l$ and $p_g$ are the limits from Lemma~\ref{conv_eps} .
 Then, using weak lower semicontinuity of norms, at the limit we get
\begin{align*}
  \frac{1}{\delta t}\int_{\Omega}\Phi [  {\cal E}(p_l, p_g) -  {\cal E}(p_l^*, p_g^*) ]\, dx + 
 \int_{\Omega} [ |\nabla p|^2   &+    |\nabla \beta(S)|^2 + |\nabla u|^2  ]\, dx \\ 
& + \eta \int_{\Omega} |\nabla p_g - \nabla p_l|^2 \, dx \leq  C ,\nonumber
\end{align*}
where constant $C$ do not change and stay independent of  $\delta t$ and $\eta$.
This bound can be applied to all time levels $k$, leading to
\begin{align*}
   \frac{1}{\delta t}\int_{\Omega}\Phi [  {\cal E}(p_l^k, p_g^k) -  {\cal E}(p_l^{k-1}, p_g^{k-1}) ]\, dx + 
   \int_{\Omega} [ |\nabla p^k|^2
  &+    |\nabla \beta(S^k)|^2 + |\nabla u^k|^2  ]\, dx \\ 
& + \eta \int_{\Omega} |\nabla p_g^k - \nabla p_l^k|^2 \, dx \leq  C .\nonumber
\end{align*}
 Multiplying this inequality by $\delta t$ and summing from $1$ to $M$ we obtain
\begin{align*}
  \int_{\Omega}\Phi   {\cal E}(p_l^M, p_g^M)\, dx 
  &+  \int_{Q_T} ( |\nabla p^{\delta t}|^2
  +    |\nabla \beta(S^{\delta t})|^2 + |\nabla u^{\delta t}|^2  )\, dx \\ 
& + \eta \int_{Q_T} |\nabla p_g^{\delta t} - \nabla p_l^{\delta t}|^2 \, dx \leq  C
+\int_{\Omega}\Phi {\cal E}(p_l^{0}, p_g^{0}) \, dx.\nonumber
\end{align*}
From Lemma~\ref{lema-E-est} and $p_g^0\in L^2(\Omega)$, $p_g^0\geq 0$  we get the following bound:
\begin{lemma}\label{lema-est-dt}
Let $p_l^{\delta t}$ and $p_g^{\delta t}$ be the solution to (\ref{discrete-1dt}), (\ref{discrete-2dt}) and 
let the secondary variables be denoted by $S^{\delta t}$, $u^{\delta t}$ and $p^{\delta t}$.
Then there exists a constant $C>0$, independent of  $\delta t$ and $\eta$, such that 
\begin{align}
   \int_{Q_T} ( |\nabla p^{\delta t}|^2
  +    |\nabla \beta(S^{\delta t})|^2 + |\nabla u^{\delta t}|^2  )\, dx \, dt
 + \eta \int_{Q_T} |\nabla p_g^{\delta t} - \nabla p_l^{\delta t}|^2 \, dx \, dt \leq  C.
 \label{delta-t-est}
\end{align}
\end{lemma}
Let us introduce the function 
\begin{align*}
r_g^k = \hat u(p_g^{k})S^k+\hat \rho_g(p_g^{k})(1-S^k),
\end{align*}
and corresponding piecewise constant and  piecewise linear time dependent functions which will be denoted 
by  ${r}_g^{\delta t}$ and $\tilde{r}_g^{\delta t}$, respectively.

\begin{lemma}\label{Prop2} 
Let $p_l^{\delta t}$ and $p_g^{\delta t}$ be the solution to (\ref{discrete-1dt}), (\ref{discrete-2dt}) from Theorem~\ref{TM_3}.
Then the following bounds uniform with respect to $\delta t$ hold:
     \begin{align}
	(p^{\delta t})_{\delta t} \text{ is uniformly bounded in } L^2(0,T; V),\label{Pb-1:1}\\
    (\beta(S^{\delta t}))_{\delta t} \text{ is uniformly bounded in } L^2(0,T; H^1(\Omega )),\label{Pb-1:2}\\
    (u^{\delta t})_{\delta t} \text{ is uniformly bounded in } L^2(0,T; V),\label{Pb-1:3}\\
	(p_c(S^{\delta t}))_{\delta t} \text{ is uniformly bounded in } L^2(0,T; V),\label{Pb-1:4}\\
	(S^{\delta t})_{\delta t} \text{ is uniformly bounded in } L^2(0,T; H^1(\Omega)),\label{Pb-1:5}\\
(\tilde S^{\delta t})_{\delta t} \text{ is uniformly bounded in } L^2(0,T; H^1(\Omega)),\label{Pb-1:12}\\
(p_l^{\delta t})_{\delta t} \text{ is uniformly bounded in } L^2(0,T; V),\label{Pb-1:6}\\
  (p_g^{\delta t})_{\delta t} \text{ is uniformly bounded in }  L^2(0,T; V),\label{Pb-1:7}\\
  (r_g^{\delta t})_{\delta t} \text{ is uniformly bounded in } L^2(0,T; H^1(\Omega)),\label{Pb-1:8}
    \end{align}
     \begin{align}
	(\tilde r_g^{\delta t})_{\delta t} \text{ is uniformly bounded in } L^2(0,T; H^1(\Omega)),\label{Pb-1:11}\\
	(\Phi\partial_t\tilde{S}^{\delta t})_{\delta t} \text{ is uniformly bounded in } L^2(0,T; H^{-1}(\Omega)),\label{Pb-1:9}\\
	(\Phi\partial_t\tilde{r}_{g}^{\delta t})_{\delta t} \text{ is uniformly bounded in } L^2(0,T; H^{-1}(\Omega)).\label{Pb-1:10}
    \end{align}
\end{lemma}
\begin{proof} 
Estimates (\ref{Pb-1:1}), (\ref{Pb-1:2}), (\ref{Pb-1:3}), (\ref{Pb-1:4}) are consequences of (\ref{delta-t-est}). 
Using \Hb{4} we get
\begin{align*}
\eta \int_{Q_T}|\nabla(p_g^{\delta t} - p_l^{\delta t} )|^2\, dx \, dt =
	\eta \int_{Q_T}|\nabla(p_c^{\delta t})|^2\, dx \, dt
	 \geq M_0^2 \eta \int_{Q_T} |\nabla S^{\delta t}|^2\, dx \, dt, 
\end{align*}
and estimate (\ref{Pb-1:5}) follows from  (\ref{delta-t-est}).
 Estimate (\ref{Pb-1:6}) is a consequence of  (\ref{glrelation1-grad}) and estimates (\ref{Pb-1:1}) and (\ref{Pb-1:4}).
 Estimate (\ref{Pb-1:7}) for  $p_g^{\delta t}$ follow from the boundedness of the regularizing term in  (\ref{delta-t-est}).

  From definition of function $r_g^{\delta t}$ we have
 \begin{align*}
 	\nabla r_g^{\delta t}  &= \sum_{k=1}^{M}\left(  S^k \nabla u^k + ({u}^k -\hat{\rho}_g(p_g^k))\nabla S^k + \hat{\rho}_g'(p_g^k)(1 - S^k)\nabla p_g^k \right)
	\chi_{(t_{k-1},t_k]}(t) 
 \end{align*}
Due to the fact  that $\hat{\rho}_g$, $\hat{u}$  and $\hat{\rho}_g^\prime$ are bounded functions we conclude
\begin{align*}
	&\lVert \nabla r_g^{\delta t} \lVert^2_{L^2(Q_T)}
	 \leq C(\lVert \nabla u^{\delta t} \lVert^2_{L^2(Q_T)} + \lVert \nabla p_g^{\delta t} \lVert^2_{L^2(Q_T)} 
	 + \lVert \nabla S^{\delta t} \lVert^2_{L^2(Q_T)}), 
\end{align*}
where constant $C$ does not depend on $\delta t$. Applying  (\ref{Pb-1:3}), (\ref{Pb-1:5}) and (\ref{Pb-1:7}) we get estimate (\ref{Pb-1:8}). 
From definitions of functions $\tilde S^{\delta t}$ and $\tilde{r}_g^{\delta t}$, and the fact that $p_g^0,p_l ^0\in H^1(\Omega)$,
we have
\begin{align*}
	&\lVert \nabla \tilde S^{\delta t} \lVert^2_{L^2(Q_T)} \leq C( \lVert \nabla S^{\delta t} \lVert^2_{L^2(Q_T)} + \lVert \nabla S^{0} \lVert^2_{L^2(\Omega)}), \\
	&\lVert \nabla \tilde r_g^{\delta t} \lVert^2_{L^2(Q_T)} \leq C(\lVert \nabla p_g^{\delta t} \lVert ^2_{L^2(Q_T)} + \lVert \nabla S^{\delta t} \lVert^2_{L^2(Q_T)} + \lVert \nabla p_g^{0} \lVert^2_{L^2(\Omega)} + \lVert \nabla S^{0} \lVert^2_{L^2(\Omega)}),
\end{align*}
and therefore we obtain estimates (\ref{Pb-1:12}) and (\ref{Pb-1:11}).
The estimates (\ref{Pb-1:9}) and (\ref{Pb-1:10}) follow from (\ref{Pb-1:3}) - (\ref{Pb-1:8}) and variational equations 
(\ref{discrete-1dt}) and (\ref{discrete-2dt}).
 \end{proof}

\subsection{End of the proof of Theorem~\ref{TM2}}
\label{sec-TM2}

In this section we pass to the limit as $\delta t\to 0$.
 
\begin{proposition}\label{Prop4}
 Let  \Hb{1}-\Hb{8} hold     and assume $(p_{l}^0, p_g^{0})\in H^1(\Omega)\times H^1(\Omega)$, $p_g^0\geq 0$.  
Then there is subsequence, still denoted $(\delta t)$, such that the following convergences hold when $\delta t$ goes to zero:
\begin{align}
	&S^{\delta t} \to S \text{ strongly in } L^2(Q_T) \text{ and a.e. in } Q_T,\label{P2-c3}\\
	&\beta(S^{\delta t}) \rightharpoonup \beta(S) \text{ weakly in }  L^2(0,T; H^1(\Omega)) \text{ and a.e. in } Q_T,\label{P2-c4}\\
	&p^{\delta t} \rightharpoonup p \text{ weakly in }  L^2(0,T; V) ,\label{P2-c5}\\
	&p_l^{\delta t} \rightharpoonup p_l \text{ weakly in }  L^2(0,T; V), \label{P2-c7}\\
	&p_g^{\delta t} \rightharpoonup p_g \text{ weakly in }   L^2(0,T; V)  \text{ and a.e. in } Q_T, \label{P2-c8}\\
	&u^{\delta t} \rightharpoonup u =\hat{u}(p_g) \text{ weakly in }  L^2(0,T; V),\label{P2-c9}\\
	&{r}_{g}^{\delta t} \to \hat{u}(p_g)S + \hat{\rho}_g(p_g)(1-S) \text{ strongly in }  L^2(Q_T) \text{ and a.e. in } Q_T.\label{P2-c6}
\end{align}
Furthermore, $0\leq S\leq 1 $, 
   and
 \begin{align}
	\Phi \partial_t \tilde{S}^{\delta t} \rightharpoonup
	\Phi \partial_t S\label{P2-c5-1}
	\quad \text{ weakly in }  L^2(0,T; H^{-1}(\Omega)),\\
	\Phi \partial_t \tilde{r}_{g}^{\delta t} \rightharpoonup
	\Phi \partial_t(\hat\rho_g(p_g)(1-S)+ \hat{u}(p_g) S)\label{P2-c5-2}
	\quad \text{ weakly in }  L^2(0,T; H^{-1}(\Omega)).
    \end{align}
\end{proposition}
\begin{proof}
From estimates (\ref{Pb-1:9}) and (\ref{Pb-1:12}) we conclude that $(\tilde{S}^{\delta t})$ is relatively compact in $L^2(Q_T)$
and one can extract a subsequence  converging strongly in $L^2(Q_T)$ and a.e. in $Q_T$ to some $S\in L^2(Q_T)$. 
Obviously we have $0\leq S\leq 1$. By applying Lemma~3.2 from \cite{schweizer1} 
we find (\ref{P2-c3}).
The weak convergences in (\ref{P2-c4}), (\ref{P2-c5}),  (\ref{P2-c7}), (\ref{P2-c8}) and (\ref{P2-c9}) follow from  Lemma~\ref{Prop2}.

The estimates (\ref{Pb-1:11}) and (\ref{Pb-1:10}) give relative compactness of the sequence $(\tilde{r}_{g}^{\delta t})_{\delta t}$
and, on a subsequence,
$$ \tilde{r}_{g}^{\delta t} \to r_g \text{ strongly in }  L^2(Q_T) \text{ and a.e. in } Q_T. $$
By applying Lemma~3.2 from \cite{schweizer1}  we also have the convergence 
$$ {r}_{g}^{\delta t} \to r_g \text{ strongly in }  L^2(Q_T) \text{ and a.e. in } Q_T. $$ 
It remains to show that $r_g = \hat{u}(p_g)S + \hat{ \rho}_g(p_g)(1 - S)$. 
From assumptions \Hb{5} and \Hb{6} we have for any $v\in L^{2}(Q_T)$
\begin{align*}
	\int_{Q_T} \Big( \hat{u}(p_g^{\delta t} &) S^{\delta t} + \hat{\rho}_g(p_g^{\delta t})  ( 1 - S^{\delta t}) \\ 
	&- [\hat{u}(v) S^{\delta t} + \hat{\rho}_g(v)  ( 1 - S^{\delta t}) ]\Big) ( p_g^{\delta t} - v ) \, dx dt \geq 0 .
\end{align*}
After passing to the limit $\delta t \to 0$ we obtain for all $v\in L^2(Q_T)$,
\begin{align*}
\int_{Q_T} \Big( r_g - [\hat{u}(v) S + \hat{\rho}_g(v)  ( 1 - S) ] \Big) ( p_g - v ) \, dx \, dt \geq 0.
\end{align*}
 By setting $v = p_g - \sigma v_1$ and passing to the limit $\sigma \to 0$ we get for all $v_1 \in L^2(Q_T)$:
 \begin{align*}
   \int_{Q_T} \Big( r_g - [\hat{u}(p_g) S + \hat{\rho}_g(p_g)  ( 1 - S) ] \Big) v_1 \, dx \, dt \geq 0,
 \end{align*}
 which gives $r_g = \hat{u}(p_g)S + \hat{\rho}_g(p_g)(1 - S)$ and (\ref{P2-c6}) is proved.  
 Then obviously we also have $\hat{u}(p_g^{\delta t})S + \hat{\rho}_g(p_g^{\delta t})(1-S) \to \hat{u}(p_g)S + \hat{\rho}_g(p_g)(1-S)$
 a.e. in $Q_T$. 
 Since the functions $\hat{u}$ i $\hat{\rho}_g$ are $C^1$ {increasing} functions we have
 \begin{align*}
      \hat{u}'(p_g^{\delta t})S + \hat{\rho}'_g(p_g^{\delta t})(1 - S) > 0,
 \end{align*}
 which gives $p_g^{\delta t} \to p_g$ a.e. in $Q_T$. Consequently we conclude that $u=\hat{u}(p_h)$.
Convergences (\ref{P2-c5-1}) and (\ref{P2-c5-2}) are consequences of estimates (\ref{Pb-1:9}), (\ref{Pb-1:10})
and (\ref{P2-c6}). 
\end{proof}


Using the convergence results in Proposition~\ref{Prop4} and the boundedness of all nonlinear
coefficients, we can
 now pass to the limit as $\delta t\to 0$ in the variational equations (\ref{discrete-1dt}), (\ref{discrete-2dt})
and find that, for all $\varphi, \psi\in L^2(0,T; V)$ equations (\ref{TM2-ve1}) and (\ref{TM2-ve2}) hold.

Let us denote $r_g = \hat\rho_g(p_g)(1-S)+ \hat{u}(p_g) S$. Then, from $S, r_g \in L^2(0,T; H^1(\Omega))$   and
$\Phi\partial_t S, \Phi\partial_t r_g\in L^2(0,T; H^{-1}(\Omega))$ it follows immediately that
 $S,\ r_g\in  C([0,T]; L^2(\Omega)).$ By standard technique, using integration by parts, we see that the 
initial conditions, $ S(0) = S^{0}$ and  $ r_{g}(0)=r_{g}^{0}$  are satisfied a.e. in $\Omega$ at $t=0$.
Finally, nonnegativity of the gas pseudo-pressure, $p_g \geq 0$, follows from the pointwise convergence.
This concludes the proof of Theorem~\ref{TM2}.

\section{Proof of the Theorem~\ref{TM1}}
\label{sec:proof-1}

Theorem~\ref{TM1} will be proved by passing to the limit as $\eta \to 0$ in the regularized problem 
(\ref{TM2-ve1}), (\ref{TM2-ve2}). We
now denote explicitly the dependence of the regularized solution on the parameter $\eta$. 
In order to apply Theorem~\ref{TM2}, we will regularize the initial conditions 
$p_l^0, p_g^{0}\in L^2(\Omega)$ with the regularization parameter $\eta$ and denote the regularized initial conditions by 
$p_l^{0,\eta}, p_g^{0,\eta}\in H^1(\Omega)$. We assume that $p_l^{0,\eta}\to p_l^0$ and  $p_g^{0,\eta}\to p_g^{0}$
in $L^2(\Omega)$ and a.e. in $\Omega$, when $\eta$ tends to zero.

As before we introduce notation
\begin{align}
r_g^{\eta}=\hat{u}(p_g^{\eta})S^{\eta}+\hat \rho_g(p_g^{\eta})(1-S^{\eta}).
\end{align}
By passing to the limit $\delta t \to 0$ to the estimate (\ref{delta-t-est}) and using weak lower semi-continuity of 
the norms we find 
\begin{align}
	   \int_{Q_T} ( |\nabla p^{\eta}|^2
	   +    |\nabla \beta(S^{\eta})|^2 + |\nabla u^{\eta }|^2  )\, dx \, dt
	   + \eta \int_{Q_T} |\nabla p_g^{\eta} - \nabla p_l^{\eta}|^2 \, dx \, dt \leq  C,
	   \label{eta-est}
\end{align}
where $C>0$ is independent of $\eta$. From this estimate we obtain following bounds with respect to $\eta$:
\begin{align}
&	(p^{\eta})_{\eta} \;\text{ is uniformly bounded in } L^2(0,T; V), \label{L4.1-1}\\
&	(u^{\eta})_{\eta} \;\text{ is uniformly bounded in } L^2(0,T; V), \label{L4.1-2}\\
&	(\beta^\eta(S^\eta))_{\eta} \;\text{ is uniformly bounded in } L^2(0,T; H^1(\Omega)), \label{L4.1-3}\\
&       (\sqrt{\eta} \nabla p_c(S^\eta))_{\eta} \;
\text{ is uniformly bounded in } L^2(Q_T)^d ,\label{L4.1-5}\\ 
&	(\Phi \partial_t(S^\eta))_\eta
\text{ is uniformly bounded in } L^2(0,T; H^{-1}(\Omega)),\label{L4.1-6}\\
&	(\Phi \partial_t(r_g^{\eta} ))_\eta
\text{ is uniformly bounded in } L^2(0,T; H^{-1}(\Omega)).\label{L4.1-7}
\end{align}
 Through the limit process are also conserved the following estimates:
\begin{align}
& 0\leq S^{\eta}\leq 1\; \text{ almost everywhere in } Q_T, \label{L4.1-8}\\
& p_g^{\eta}\geq 0\; \text{ almost everywhere in } Q_T. \label{L4.1-9}
\end{align}
Due to Lemma~\ref{lema-GP-bounds} and (\ref{L4.1-9}) we also have
\begin{align}
&	(p_g^{\eta})_{\eta} \;\text{ is uniformly bounded in } L^2(Q_T). \label{L4.1-new}
\end{align}

For passage to the limit as $\eta\to 0$ we need  compactness in $L^2(Q_T)$ of sequences
$(S^\eta)$ and $(r_g^\eta)$ which will be proved by an application to Lemma~4.2 in \cite{AAPP}. 
Therefore, we need the following estimates: 
\begin{lemma}\label{Prop5}
	Under assumptions \Hb{1} - \Hb{8}, we have the following inequalities 
	\begin{align}
	\int_{Q_T} \lvert S^{\eta}(x+\Delta x, t) - S^{\eta}(x,t) \lvert^2 \, dx \, dt \leq \omega (\lvert \Delta x \lvert)  \label{Lb3-1:1}\\
	\int_{Q_T} \lvert r_g^{\eta}(x+\Delta x, t) - r_g^{\eta}(x,t) \lvert^2\, dx \, dt \leq \tilde\omega (\lvert \Delta x \lvert) \label{Lb3-1:2},
	\end{align}
	for all $\Delta x \in \mathbb{R}^d$, where functions $\omega$ and $\tilde{\omega} $ are continuous and independent of $\eta$  and satisfy $\lim_{|\Delta x |\rightarrow 0} \omega( |\Delta x|) = 0$ and $\lim_{|\Delta x |\rightarrow 0} \tilde\omega( |\Delta x|) = 0$.
\end{lemma}
\begin{proof}
By using \Hb{8} and bound (\ref{L4.1-3}) we obtain in a standard way
\begin{align}
\int_{Q_T}  &\lvert S^{\eta}(x + \Delta x, t) - S^{\eta}(x, t)  \lvert ^2 \, dx  \, dt 
 \leq C |\Delta x|^{2\tau}, \label{Pb3-1:1a}
\end{align}
which proves (\ref{Lb3-1:1}). 
In order to obtain (\ref{Lb3-1:2}) we will consider the two parts of $r_g^\eta$ separately. 
The first part,  $\hat{u} (p_g) S$, is easy to estimate using (\ref{Pb3-1:1a}) and the bound 
(\ref{L4.1-2}). We get 
\begin{align}
\int_{Q_T} &\lvert \hat{u}(p_g^{\eta}( x+\Delta x, t )) S^{\eta}(x + \Delta x, t)) - \hat{u} (p_g^{\eta}(x,t))S^{\eta}(x,t) \lvert^2 \, dx \, dt \leq C( | \Delta x|^{2} + |\Delta x |^{2\tau})
\end{align}
The second term $(1-S^{\eta})\hat{\rho}_g(p_g^{\eta})$ can be written as 
$(1-S^{\eta})\hat{\rho}_g(p^{\eta} -\hat{P}(S^\eta))$ in the whole domain $Q_T$ since $1-S^{\eta}$
is equal to zero in the one phase region.     We have, 
\begin{align*}
\int_{Q_T} &| (1-S^{\eta})(x+\Delta x, t ) \hat{\rho}_g(p_g^{\eta}( x+\Delta x,t )) - (1-S^{\eta})(x,t)\hat{\rho}_g(p_g^{\eta}(x,t)) |^2 \, dx \, dt \\
& \leq  \int_{Q_T} |(1-S^{\eta})(x+\Delta x, t ) \left( \hat{\rho}_g(p_g^{\eta}(x+ \Delta x, t)) - \hat{\rho}_g(p_g^{\eta}(x,t)) \right) |^2 \, dx \, dt  \\
& +  \int_{Q_T} |(S^{\eta}(x+\Delta x, t) - S^{\eta}(x,t)) \hat{\rho}_g(p_g^{\eta}(x,t)) |^2 \, dx \, dt
\end{align*}
The second term on the right-hand side is estimated by using (\ref{Pb3-1:1a}) and the boundedness of the function $\hat{\rho}_g$.
In order to estimate the first term on the right hand side we first note that by \Hb{6} the function $\hat{\rho}_g$
has bounded derivative. Then we can estimate,
\begin{align*}
\int_{Q_T} &|(1-S^{\eta})(x+\Delta x, t ) \left( \hat{\rho}_g(p_g^{\eta}(x+ \Delta x, t)) - \hat{\rho}_g(p_g^{\eta}(x,t)) \right) |^2 \, dx \, dt \\
& \leq C \int_{Q_T}  | p^{\eta}( x + \Delta x, t ) - p^{\eta}( x,t ) |^2 \, dx \, dt  \\
& + C \int_{Q_T} |(1-S^{\eta})(x+\Delta x, t) \hat{P}( S^{\eta}(x + \Delta x, t)) - (1-S^{\eta})(x, t)\hat{P}(S^{\eta}(x, t)) |^2 \, dx \, dt \\
& + C \int_{Q_T} |(S^{\eta}(x+\Delta x, t) -S^{\eta})(x, t))  \hat{P}(S^{\eta}(x, t)) |^2 \, dx \, dt
\end{align*}
The first integral on the right hand side is estimated due to (\ref{L4.1-1}), and the estimate for the third 
integral follows from the boundedness of the function $\hat{P}(S)$ and the bound  (\ref{Pb3-1:1a}). 
The second integral on the right hand side is estimated using  assumption \Hb{9} which finally leads to an estimate:
\begin{align*}
\int_{Q_T} \lvert r_g^{\eta}(x+\Delta x, t) - r_g^{\eta}(x,t) \lvert^2 
\leq C( |\Delta x|^2 + |\Delta x|^{2\tau} + |\Delta x|^{2\overline{\tau}}),  
\end{align*}
and (\ref{Lb3-1:2}) is proved.
\end{proof}

\begin{lemma}
	(Strong and weak convergences) Up to subsequence the following convergence results hold:
	\begin{align}
	& p^{\eta} \rightharpoonup p  \;\text{ weakly in } L^2(0,T; V), \label{L4.3-1} \\
	&\beta(S^{\eta}) \rightharpoonup \beta(S)  \;\text{ weakly in } L^2(0,T; H^1(\Omega)),\label{L4.3-2}\\
	&S^{\eta}\to S \;\ \text{strongly in } L^2(Q_T) \text{ and a.e. in } Q_T,\label{L4.3-4}\\
	&r_g^{\eta}\to \hat u(p_g) S+\hat\rho_g(p_g)(1-S)  \;\ \text{strongly in } L^2(Q_T) \text{ and a.e. in } Q_T,\label{L4.3-5}\\
	& p_g^{\eta}\rightarrow p_g \; \text{weakly and a.e in } Q_T, \label{L4.3-6}\\
	&u^{\eta}\rightharpoonup \hat{u}(p_g)  \;\text{ weakly in } L^2(0,T; V),\label{L4.3-3}\\
	& p_l^{\eta}\rightarrow p_l \; \text{a.e in } Q_T, \label{L4.3-7} \\
	&\Phi \partial_t  S^{\eta}  \rightharpoonup
	\Phi \partial_t S \text{ weakly in } L^2(0,T; H^{-1}(\Omega)), \label{L4.3-8}\\
	&\Phi \partial_t r_g^{\eta} \rightharpoonup \Phi \partial_t(\hat u(p_g) S+\hat\rho_g(p_g)(1-S))
	\;\text{ weakly in } L^2(0,T; H^{-1}(\Omega)) .\label{L4.3-9}
	\end{align}
	\label{lemma-4.3}
\end{lemma}
\begin{proof}
If we apply Lemma~4.2 from \cite{AAPP} to the estimates (\ref{Lb3-1:1}), (\ref{L4.1-8}) and (\ref{L4.1-6}) we obtain
\begin{align*}
		S^{\eta}\to S  \;\ \text{strongly in } L^2(Q_T) \text{ and a.e. in } Q_T.
\end{align*}
In the same way the boundedness of $r_g^{\eta}$, estimates (\ref{Lb3-1:2}) and (\ref{L4.1-7}) imply 
\begin{align*}
	r_g^{\eta}\to r_g  \;\ \text{strongly in } L^2(Q_T) \text{ and a.e. in } Q_T.
\end{align*}
We can extract a subsequence such that $p_g^{\eta}\rightharpoonup p_g$ weakly in $L^2(Q_T)$
and then by using the monotonicity argument, as in Proposition~\ref{Prop4},  we find that $r_g = \hat u(p_g) S+\hat\rho_g(p_g)(1-S)$ and  
 obtain the convergence:
\begin{align*}
p_g^{\eta}\to p_g  \text{ a.e. in } Q_T.
\end{align*}
All other convergences  follow immediately from the bounds (\ref{L4.1-1})--(\ref{L4.1-7}). 
\end{proof}

By using convergence results from previous proposition combined with the boundedness of nonlinear coefficients, 
equality (\ref{fund-gl1}) and estimate (\ref{L4.1-5}) we can pass to the limit $\eta \to 0$ 
in equations (\ref{TM2-ve1}) and (\ref{TM2-ve2}) to obtain variational equations (\ref{VF-1}) and (\ref{VF-2}).
Passing to the limit $\eta \to 0$ in inequality $p_g^{\eta} \geq 0$ we find  $p_g \geq 0$ a.e. in $Q_T$. 
Using an integration by parts in the regularized $\eta$-problem and the limit problem, with the test function of the form 
$\psi(x)\varphi(t)$, $\psi\in V$ and $\varphi\in C^1([0,T])$ with  $\varphi(0)=1$, $\varphi(T)=0$ we find in a standard way that
the initial conditions (\ref{VF-3}), (\ref{VF-4}) are satisfied.
This completes the proof of Theorem~ ~\ref{TM1}.

\section*{Acknowledgments}
This work was supported by Croatian science foundation project no 3955.


\end{document}